\documentclass[12pt]{article}
\usepackage{geometry}
\usepackage{graphicx}
\usepackage{amsmath,amssymb,amsthm}
\usepackage{enumerate}
\usepackage{hyperref}
\usepackage[active]{srcltx}
\numberwithin{equation}{section}
\usepackage[T1]{fontenc}
\usepackage{color}
\newtheorem{theo}{Theorem}
\newtheorem{lem}{Lemma}[section]

\newtheorem{prop}{Proposition}[section]
\newtheorem{rmk}{Remark}[section]
\newcommand{\eps}{\varepsilon}
\newcommand{\z}{\zeta}
\newcommand{\s}{\sigma}
\newcommand{\R}{\mathbb{R}}

\renewcommand{\o}{\overline}

\begin{document}
\title{Linear relaxation to planar Travelling Waves in Inertial Confinement Fusion}
\author{
L. Monsaingeon
\footnote{
Institut de Math\'ematiques de Toulouse, Universit\'e Paul Sabatier, 118 route de Narbonne, 31062 Toulouse, FRANCE
\href{mailto:leonard.monsaingeon@math.univ-toulouse.fr}{\nolinkurl{leonard.monsaingeon@math.univ-toulouse.fr}}
}
}
\date{}
%
%
\maketitle

\begin{abstract}
\noindent
We study linear stability of planar travelling waves for a scalar reaction-diffusion equation with non-linear anisotropic diffusion. The mathematical model is derived from the full thermo-hydrodynamical model describing the process of Inertial Confinement Fusion. We show that solutions of the Cauchy problem with physically relevant initial data become planar exponentially fast with rate $s(\eps',k)>0$, where $\eps'=\frac{T_{min}}{T_{max}}\ll 1$ is a small temperature ratio and $k\gg 1$ the transversal wrinkling wavenumber of perturbations. We rigorously recover in some particular limit $(\eps',k)\rightarrow (0,+\infty)$ a dispersion relation $s(\eps',k)\sim \gamma_0 k^{\alpha}$ previously computed heuristically and numerically in some physical models of Inertial Confinement Fusion.
\end{abstract}

\section{Introduction}
The aim of Inertial Confinement Fusion - ICF in short - is to concentrate extreme temperature and pressure conditions in plasma to a single point in order to trigger a nuclear fusion reaction. The corresponding models couple thermal reaction-diffusion equations with hydrodynamical effects. For such extreme temperature regimes (\mbox{$T\sim 10^7K$}) the heat conductivity coefficient $\lambda$ in the usual diffusion term $\nabla\cdot(\lambda \nabla T)$ cannot be considered as constant, but must be taken of the form
$$
\lambda=\lambda(T)=T^{m-1}
$$
for some fixed conductivity exponent $m>1$ (the value $m=7/2$ is often considered in ICF, corresponding to the \textit{Spitzer electronic heat conductivity}). Scaling at the hot side we always take in the sequel temperature values  $\eps'=T_{min}<T<T_{max}=1$, and consider the limit of small temperature ratio $\eps'\ll 1$ (which is indeed true in ICF, typically $\frac{T_{min}}{T_{max}}\sim 10^{-3}$). An important consequence is here the strong variation of the diffusion length scale: since temperature varies of several orders of magnitude so does the thermal conductivity $\lambda=T^{m-1}$, and the effective diffusion length scales are thus very different in the cold and hot regions. The physically relevant wavenumbers are therefore $1\ll k\ll 1/(\eps')^{m-1}$, corresponding to wavelengths $\frac{1}{k}$ between the shortest and longest diffusive length scales. The second important feature of ICF is the propagation of spherical temperature waves, which concentrate some laser driven input of energy to the center of a fuel target. These waves have a very particular spatial structure consisting in three separate zones: in the most inner one (the target) an initial amount of unburned fuel is at rest $T\approx \eps'$, in the outer one the fuel has been transformed into hot plasma by the laser $T\approx 1$, and in the intermediate region temperature evolves between $\eps'$ and $1$ according to a specific power law.  
\par
When the temperature ratio $\eps'$ is small a very thin boundary layer, called \textit{ablation front}, separates in the wave the cold region from the intermediate one. In the theoretical situation the geometry is purely radial and the ablation front should be a sphere shrinking in time towards its center. This ideal scenario is however unstable due to \textit{Rayleigh-Taylor} and \textit{Darrieus-Landau} effects, which are purely hydrodynamical and concentrate at the ablation front: should the latter slightly deviate from its ideal spherical geometry, wrinkling may amplify, centripetal waves may stop propagating, and the desired fusion reaction may not be triggered. A so-called \textit{transversal mass ablation effect} fortunately tends to stabilize the system. This phenomenon can be roughly explained as follows: whereas the diffusion in the radial direction is responsible for the particular structure of the wave, the orthoradial diffusion across the wrinkled front simply tends to reduce its unstable transversal oscillations.
\par
A self-consistent linear analysis of the full thermo-hydrodynamical model \cite{SanzClavinMasse-LinearDLRTinstab}, performed in the relevant frequency regime for large conductivity exponents $m\rightarrow +\infty$, showed that the transversal mass ablation should lead to a stabilizing contribution
$$
s(\eps',k)\sim \nu k^{1-\frac{1}{m-1}},\qquad \nu=cst>0
$$
in the full dispersion relation ($s>0$ corresponding to an exponential \textit{damping} rate). It was suggested in \cite{ClavinMasse-instability} that this stabilization mechanism can actually be investigated looking at a much simpler model, namely the purely diffusive relaxation of thermal waves, free from hydrodynamical effects. A first rigorous study of an approximated non-Galilean model \cite{ClavinMasseRoque-relaxation} for finite conductivity exponents led to a very similar dispersion relation in the regime $1\ll k^2\ll 1/(\eps')^{m-1}$, except for a logarithmic correction. In this paper we derive a purely thermal model suitably approximating the full one, and rigorously recover the anticipated self-consistent dispersion relation $s\sim \nu_m k^{1-\frac{1}{m-1}}$ for finite conductivity exponents $m>3$ encompassing the physical case $m=7/2$ (in some frequency regime discussed later and relating $k\gg 1$ to $\eps'\ll 1$).
%
%
\section{Model and contents}
The spherical geometry being difficult to investigate we perform our study for the following planar configuration: $x\in\R$ will denote the infinite longitudinal direction of propagation (mimicking the radial depth of penetration in the target for the real setting) and $y\in\R$ the transversal one. This planar approximation is legitimate in ICF since the relevant wavelengths are much smaller than the radius of the target. Neglecting the transversal velocity $\overrightarrow{V}=(V,0)$ we consider now the following 2 dimensional thermo-hydrodynamical model for $(t,x,y)\in\R^3$
\begin{subequations}
\label{eq:modelelabo}
\begin{align}
\rho\partial_tT-\nabla\cdot(\lambda\nabla T)+\rho V\partial_xT=&\;f(T),
\label{eq:energie}\\
\partial_t\rho+\partial_x(\rho V) = &\; 0,\label{eq:masse}\\
\rho T  = &\; 1,
\label{eq:quasiisobare}\\
T(t,-\infty,y)=\eps' \quad & \quad T(t,+\infty,y)=1.
\label{eq:CLsvariablex}
\end{align}
\end{subequations}
Here $T$ is the reduced temperature of the plasma scaled at the hot side ($T_{min}=\eps'\ll 1 =T_{max}$), $\rho$ its density, $V\in\R$ its longitudinal velocity, $\lambda=\lambda(T)=T^{m-1}$ the reduced non-linear heat conductivity, and $f(T)$ a non-linear reaction term of ignition type modelling input of energy by the laser in the outer region. Equations \eqref{eq:energie}-\eqref{eq:masse} correspond to conservation of energy and mass, and \eqref{eq:quasiisobare} is a quasi-isobaric approximation $p\approx cst$ which is common in ICF. Since we neglected the transversal velocity conservation of momentum is not required here to close the system, which is however Galilean invariant. Boundary conditions \eqref{eq:CLsvariablex} simply state that the medium is at rest at $x=-\infty$ and totally combusted at $x=+\infty$. Temperature waves propagate now in the negative $x$ direction (the hot medium invades the cold one), and the transversal wrinkling of the front will be later taken into account considering a periodicity condition in the $y$ direction.
\par
Considering \eqref{eq:masse} as a Schwarz condition for crossed derivatives we may define $X(t,x,y)$ such that $\partial_t X=-\rho V$ and $\partial_xX=\rho$. More explicitly, this new Lagrangian coordinate $X$ is mass-weighted with respect to the Eulerian one $x$ as
\begin{equation}
X(t,x,y):=\int_0^x\rho(t,z,y)dz-\int_0^t\rho V(\tau,0,y)d\tau.
\label{eq:varlagr}
\end{equation}
One easily checks that for physical $\rho\geq 1/T\geq 1$ \eqref{eq:varlagr} defines a proper change of coordinates $(t,x,y)\rightarrow (t,X,y)$ such that $X(t,\pm\infty,y)=\pm \infty$, and inverting the Jacobian yields the usual characteristics $\frac{\partial x}{\partial t}=V(t,x,y)$ for Lagrangian particles. The ideal spherical geometry corresponds here to planar longitudinal waves $\partial_y=0$: from now on we use the
\begin{equation}
\text{Lagrangian approximation:}\qquad \partial_y X=0.
\label{eq:lagrangian_approx}
\end{equation}
Since coupling of the conservation of mass and energy occurs precisely through this $\partial_yX$ term this is of course a drastic approximation, but however more subtle than neglecting each and every transversal variation $\partial_y=0$ directly in the equations: transversal diffusion will still be effective and indeed induce the stabilization mechanism as explained in the introduction. Throughout the rest of the paper we will work in these Lagrangian coordinates and write $x$ instead of $X$ with a clear abuse of notations, and in the Lagrangian approximation \eqref{eq:lagrangian_approx} conservation of energy \eqref{eq:energie} is uncoupled from the hydrodynamics as
\begin{equation}
\rho\partial_{t}T-\rho\partial_x(\lambda\rho\partial_x T)-\partial_{y}(\lambda\partial_{y}T)=f(T).
\label{eq:energielagr}
\end{equation}
Since $\rho=1/T$ and $\lambda=T^{m-1}$ \eqref{eq:energielagr} is a purely scalar reaction-diffusion equation, in which it is worth noting the clear difference between longitudinal and transversal diffusions.
\par
As previously mentioned, temperature in the wave profile evolves according to a specific power law. More precisely, $\nu=T^{m-1}$ behaves linearly with respect to the Eulerian longitudinal coordinate in this intermediate region: this is very similar to the Porous Media Equation where the pressure unknown evolves linearly at least in the vicinity of the free boundary, see e.g. \cite{CaffarelliVazquezWolanski-lipschitzPME}. However, the Lagrangian coordinates are here mass-weighted and stretched with respect to the Eulerian ones according to \eqref{eq:varlagr}: the suitable variable turns to be
$$\mu:=T^{m-2}$$
which we will show to behave linearly with respect to the Lagrangian coordinate at least in the intermediate region (existence and qualitative properties of the wave solution will be investigated in Section~\ref{section:ondeplane}). Since we are concerned with wave propagation it will also be more convenient to work in the moving frame $x+ct$, in which conservation of energy \eqref{eq:energielagr} and boundary conditions \eqref{eq:CLsvariablex} read in terms of the unknown $\mu=T^{m-2}$ as
\begin{subequations}
\begin{align}
&\partial_t\mu+c\partial_x\mu-\left(\mu\partial_{xx}\mu+\frac{1}{m-2}(\partial_x\mu)^2\right)-\left(\mu^{\frac{m}{m-2}}\partial_{yy}\mu+\frac{2}{m-2}\mu^{\frac{2}{m-2}}(\partial_y\mu)^2\right)=G(\mu)\label{eq:E}\\
&\mu(t,-\infty,y)=\eps,\qquad \mu(t,+\infty,y)=1,\qquad \eps:=(\eps')^{m-2}\ll 1.
\label{eq:CLEinfty}
\end{align}
\label{eq:modele}
\end{subequations}
The right-hand side $G(\mu):=(m-2)\mu f(\mu^{\frac{1}{m-2}})$ corresponds to the initial reaction term $f(T)$ modelling the laser, and is again of ignition type: we will always assume in the following that $G$ is smooth and that there exists a threshold $\theta\in]0,1[$ such that
$$
\begin{array}{rcl}
\mu\in[0,\theta]:&\quad & G\equiv 0,\\
\mu\in]\theta,1[ :&  & G>0,\\
\mu=1 :& & G=0\;\text{ and }G'(1)<0.
\end{array}
\label{eq:cutoffG}
$$
\par
In order to investigate linear stability with respect to wrinkled perturbations $U(t,x,y)$ we linearize in the moving framce at the planar travelling wave $\mu_0(t,x,y)=p(x)$, with an additional $2\pi/k$ periodicity condition in the $y$ direction accounting for the ablation front transversal wrinkling:
\begin{subequations}
\begin{align}
&\partial_t U-\left(p\partial_{xx}U+p^{\frac{m}{m-2}}\partial_{yy}U\right)+\left(c-\frac{2p'}{m-2}\right)\partial_x U - p'' U=\frac{dG}{d\mu}(p)U,\label{eq:linearisation1}\\
&U(t,\pm\infty,y)=0,\\
&U(t,x,y+2\pi/k)=U(t,x,y).
\end{align}
\label{eq:linearisation}
\end{subequations}
Note again the difference between the longitudinal diffusion $p\partial_{xx}$ and the transversal one $p^{\frac{m}{m-2}}\partial_{yy}$ in \eqref{eq:linearisation1}. Expanding now in transversal Fourier modes $U(t,x,y)=u(x)e^{-st+iky}$ leads to
\begin{subequations}
\begin{align}
-pu''+\left(c-\frac{2p'}{m-2}\right)u'+\left(k^2p^{\frac{m}{m-2}}-p''-\frac{dG}{d\mu}(p)\right)u&=\; su,\label{eq:Elinx}\\
u(\pm\infty) & =\; 0,\label{eq:CLlinx}
\end{align}
\label{eq:modelelin}
\end{subequations}
where $'=d/dx$.
\par
Let us stress at this point that the wave speed $c=c_{\eps}>0$ and profile $p=p_{\eps}(x)>0$ are uniquely determined by the small parameter $\eps>0$ (this should be no surprise since the reaction term is of ignition type, see Section~\ref{section:ondeplane}). We will see that the planar wave satisfies $\displaystyle{\lim_{x\to -\infty}\,p(x)}=\eps$ with exponential rate $r=c_{\eps}/\eps$: when $\eps\rightarrow 0$ we have $r\rightarrow +\infty$ and the wave degenerates into a free boundary solution, $p_{\eps}\rightarrow p_0$ and $p_0(x)\equiv 0$ for $x$ sufficiently negative. This is illustrated in Figure~\ref{fig:CVpeps} and is again very similar to the Porous Media Equation scenario. Physically relevant perturbations should therefore not disturb the reference wave ahead of the front, and as a consequence we only investigate perturbations with \textit{maximal decay} $u(-\infty)=0$.
\par
For fixed temperature ratio $\eps>0$ and wavenumber $k>0$ only some particular values of $s=s(\eps,k)$ allow the connection $u(-\infty)=u(+\infty)=0$ with maximal decay on the left side, and since we set $s$ to be the \textit{damping} rate we also want to compute the smallest such $s$ (corresponding to optimal stability). Since the wave profile will satisfy $p(x)\geq \eps>0$, problem \eqref{eq:modelelin} is uniformly elliptic on the line $x\in\R$: finding the minimal $s$ is clearly a principal eigenvalue problem, and signed solutions will play an important role as usually in classical elliptic theory.
\par
Let us recall that the relevant frequency regime in ICF is $1\ll k\ll 1/(\eps')^{m-1}=1/\eps^{\frac{m-1}{m-2}}$; in the whole paper we will always assume the condition
\begin{equation}
1\ll k\ll \frac{1}{\eps^{\frac{1}{m-2}}},
\label{eq:doublelimite}
\end{equation}
which is slightly more restrictive for wavenumbers if $m>2$. For some large $k_0>0$ and small $\eta>0$ let us define the frequency (open) set
\begin{equation}
\mathcal{U}:=\Big{\{}\eps,k>0:\qquad k>k_0\text{ and } k<\frac{\eps^{\eta}}{\eps^{\frac{1}{m-2}}} \Big{\}}.
\label{eq:regime_eta}
\end{equation}
Taking $(\eps,k)\rightarrow(0,+\infty)$ in $\mathcal{U}$ clearly implies regime \eqref{eq:doublelimite}. The parameter $\eta>0$ is needed for technical reasons, but may be chosen as small as desired so that there is no real loss of information when restricting \eqref{eq:doublelimite} to $(\eps,k)\in\mathcal{U}$. Our main result is the following:
\begin{theo}
Fix a reaction term $G$ of ignition type and a conductivity exponent $m>3$: for any $\eta>0$ there exists $k_0>0$ such that, for any $(\eps,k)\in\mathcal{U}$ defined in \eqref{eq:regime_eta}, there exists a principal eigenvalue $s=s(\eps,k)>0$ such that problem \eqref{eq:modelelin} has a positive solution $u(x)$ with maximal decay when $x\to-\infty$. Moreover $s(\eps,k)>0$ is the smallest eigenvalue, there exists no other principal eigenvalue, and the associated eigenspace is one-dimensional. Finally there exists $\gamma_0=\gamma_0(m,G)>0$ such that
\begin{equation}
s(\eps,k)\sim \gamma_0 k^{1-\frac{1}{m-1}}
\label{eq:reldispasymptotvarx}
\end{equation}
when $(\eps,k)\to(0,+\infty)$ in $\mathcal{U}$.
\label{theo:reldisp}
\end{theo}
This means that the wave is linearly stable with respect to wrinkled perturbations (with maximal decay) and that any solution of the non-linear Cauchy problem with initial datum $u_0(x)$ decaying fast enough to the left will become planar exponentially fast with rate at least $s(\eps,k)>0$. Note that our result does not tell anything regarding perturbations in the longitudinal direction: invariance of the wave solutions under $x$ shift results as usual in a non-trivial element $dp/dx$ in the kernel of the linearized operator, i-e in a zero eigenvalue. For fixed $\eps>0$ non-linear techniques can be used to establish stability of the wave with respect to almost any such perturbation (see e.g.\cite{RoquejoffreMelletNolenRyzhik-stabilityTW,Roquejoffre-EventualMonotonicityTW}), but one has then to face the degeneracy into free boundary when $\eps\to 0$. The main interest of our result is that it is uniform in $\eps\to 0$, see \eqref{eq:reldispasymptotvarx}.
\par
The paper is organized as follows: in Section~\ref{section:ondeplane} we establish existence of the planar travelling wave and investigate its linear behaviour as well as the degeneracy into free boundary of the ablation front when $\eps\rightarrow 0^+$. Section~\ref{section:zonefroide} is devoted to the construction of solutions with maximal decay at $-\infty$ in the cold region $p=\mathcal{O}(\eps)$ for $s=\gamma k^{1-\frac{1}{m-1}}$ and fixed $\gamma$. This is done expanding the solution $u=u_0+\eps u_1+\eps u_2$ in the suitable scale $\xi=\frac{x}{\eps}$, and we also compute $(u,\frac{du}{dx})$ at the exit of the boundary layer. Section~\ref{section:zonechaude} is a formal limit $(\eps,k)\to(0,+\infty)$ performed in the linear region $\frac{dp}{dx}\approx cst>0$ at scale $\z=k^{1-\frac{1}{m-1}}x$. Due to the presence of the free boundary when $\eps\to 0$ the resulting limiting problem degenerates at $\z=0$, but will nonetheless yield the asymptotic coefficient $\gamma_0$ in \eqref{eq:reldispasymptotvarx} as the principal eigenvalue of some singular Sturm-Liouville problem on the half line $\z>0$. In Section~\ref{section:raccord} we rigorously justify this limit by proving that,  when $(\eps,k)\to(0,+\infty)$ in $\mathcal{U}$, the real physical setting $\eps>0,k<+\infty$ automatically matches the formal asymptotic problem $\eps=0,k=+\infty$ (the matching occurring in the ablation front).
This last section also contains the proof of Theorem~\ref{theo:reldisp}.
\section{Planar travelling wave and boundary layer}
\label{section:ondeplane}
In this section we study the aforementioned travelling wave at which the problem is linearized. Plugging $\mu(t,x,y)=p(x)$ into \eqref{eq:E} leads to
\begin{equation}
-pp''+\left(c-\frac{p'}{m-2}\right)p'=G(p), \quad p(-\infty)=\eps, \quad p(+\infty)=1,
\label{eq:EDOp}
\end{equation}
and in order to eliminate invariance under $x$ shift we impose the additional pinning condition
\begin{equation}
p(0)=\theta.
\label{eq:conditiontranslation}
\end{equation}
\begin{prop}
Let $\eps\in]0,\theta[$:
\begin{enumerate}
 \item
There exists a unique speed $c=c_{\eps}>0$ such that problem \eqref{eq:EDOp}-\eqref{eq:conditiontranslation} admits a solution $p=p_\eps(x)$. This solution is unique and satisfies $0<p_{\eps}'<(m-2)c_{\eps}$.
\item
When $\eps\rightarrow 0^+$ we have that $c_{\eps}\rightarrow c_0>0$ and $p_{\eps}(.)\to p_0(.)$ uniformly on $\R$, where $p_0$ solves
$$
\begin{array}{rcl}
x\in]-\infty,-\frac{\theta}{(m-2)c_0}] & : & p_0(x)=0\\
x\in]-\frac{\theta}{(m-2)c_0},0] & : & p_0(x)=\theta+(m-2)c_0x\\
x\in]0,+\infty] & : & \left\{
\begin{array}{l}
-p_0p_0''+\left(c_0-\frac{p_0'}{m-2}\right)p_0'=G(p_0)\\
p_0(0)=\theta,\hspace{.5cm}p_0(+\infty)=1
\end{array}
\right.
\end{array}
$$
\end{enumerate}
\label{prop:peps->p0}
\end{prop}
This is a particular case of more general well-known results in reaction-diffusion theory, see e.g. \cite{BeresLarou-qquesaspects,BerettaHulshofPeletier-ODEcoatingFlow}: the proof relies on simple ODE techniques and will only be sketched here for the sake of completeness. The degeneracy into free-boundary $p_{\eps}(x)\to p_0(x)$ is portrayed in Figure~\ref{fig:CVpeps}.
\begin{proof}
Taking advantage of the invariance of \eqref{eq:EDOp} under translations we may use the Sliding Method from \cite{BeresNiren-sliding} and see that any solution must be increasing in $x$, and therefore that $\eps<p<1$. Setting $p'=U(p)>0$, \eqref{eq:EDOp} leads to
$$
p\in]\eps,1[:\qquad -p\frac{dU}{dp}U(p)+\left(c-\frac{U(p)}{m-2}\right)U(p)=G(p)
$$
with $U(\eps)=p'(-\infty)=0$ and $U(1)=p'(+\infty)=0$. For $p\leq\theta$ the reaction term can be omitted and this can be explicitly integrated as
\begin{equation}
p\in]\eps,\theta]:\qquad U(p)=(m-2)c\left[1-\left(\frac{\eps}{p}\right)^{\frac{1}{m-2}}\right],
\label{eq:p'}
\end{equation}
and the pinning condition $p(0)=\theta$ yields 
$$
p'(0)=U(\theta)=(m-2)c\left[1-\left(\frac{\eps}{\theta}\right)^{\frac{1}{m-2}}\right].
$$
For fixed $c>0$ and given $\alpha>0$ the Cauchy problem
$$
x\geq 0:\qquad \left\{
\begin{array}{c}
 -pp''+\left(c-\frac{p'}{m-2}\right)p'=G(p)\\
p(0)=\theta\\
p'(0)=\alpha
\end{array}
\right. 
$$
is independent of $\eps$, and shooting with respect to $\alpha$ there exists a unique $\alpha=\alpha_0(c)>0$ such that the corresponding solution satisfies $p(+\infty)=1$. The solution on $x<0$ therefore matches the one on $x>0$ if and only if $(m-2)c\left(1-\left(\frac{\eps}{\theta}\right)^{\frac{1}{m-2}}\right)=p'(0)=\alpha_0(c)$. One then shows that this fixed point  equation in $c$ has a unique solution $c=c_{\eps}$ ($\alpha_0$ is decreasing in $c$) associated with a unique profile $p=p_{\eps}$, and convergence of $c_{\eps}\to c_0$ and $p_{\eps}\to p_0$ when $\eps\searrow 0$ follows by monotonicity.
\end{proof}
We establish now precise asymptotes on $p_{\eps}\to p_0$ when $\eps\to 0$. Scaling $x=\eps\xi$ and $q(\xi)=\frac{p_{\eps}(\eps \xi)}{\eps}$ \eqref{eq:p'} reads
$$
\left\{\begin{array}{c}
\frac{dq}{d\xi}=\frac{dp}{dx}=(m-2)c\left(1-\frac{1}{q^{\frac{1}{m-2}}}\right)\\
q(-\infty)=1
\end{array} \right. .
$$
This shows that $p_{\eps}$ presents a boundary layer of thickness $\eps$ in which $p=\mathcal{O}(\eps)$ (this is the ablation front), and $p'$ varies from $p'(-\infty)=0$ when $p=\eps$ (thus $q=1$) to $(m-2)c[1+o(1)]>0$ when $p\gg \eps$ (thus $q\gg 1$) as pictured in Figure~\ref{fig:CVpeps}.
\begin{figure}[h!]
\begin{center}
\scalebox{.9}{
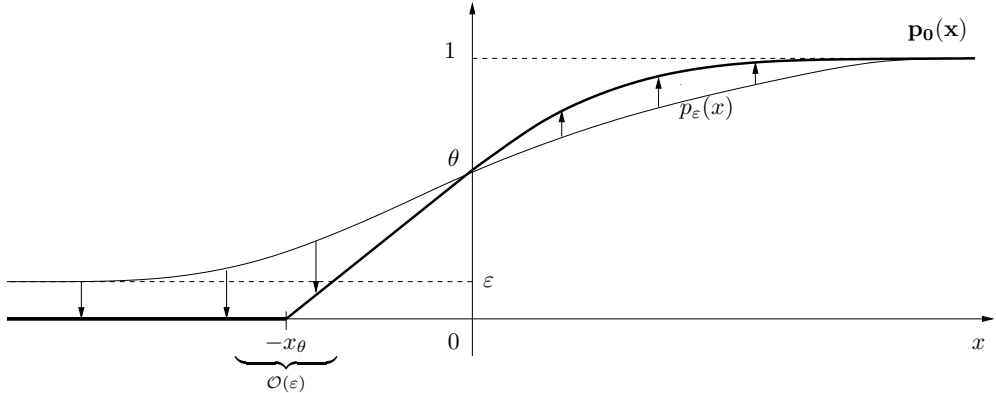
}
\end{center}
\caption{monotonic convergence $p_{\eps}\rightarrow p_0$ and $\eps$-boundary layer.}
\label{fig:CVpeps}
\end{figure}
\\
\par
From now on we choose the origin at the slope discontinuity by setting
\begin{equation}
x_{\theta}:=\frac{\theta}{(m-2)c_0}, \qquad p_0\left(x_{\theta}\right)=\theta,
\label{eq:defxthetatranslationpeps}
\end{equation}
and we are out of the boundary layer as soon as $x\gg \eps$. To the right $p_{\eps}$ is linear when $\eps\to 0$, $p_{\eps}'(x)\underset{\eps\to 0}{\sim}(m-2)c_0$ and $p_{\eps}(x)\underset{\eps\to 0}{\sim} (m-2)c_0x$. We will refer in the following to the set where $\eps\ll p<\theta$ indistinctly as the ``hot zone'' or ``linear zone'', and call ``cold zone'' the set where $p=\mathcal{O}(\eps)$ (Figure \ref{fig:CVpeps} should make this terminology self-explanatory). In addition, plugging \eqref{eq:p'} into \eqref{eq:EDOp} yields an expression of
\begin{equation}
 p_{\eps}''=\frac{1}{p_{\eps}}\left(c-\frac{U(p_{\eps})}{m-2}\right)U(p_{\eps}) =c\eps^{\frac{1}{m-2}}\frac{U(p_{\eps})}{p_{\eps}^{1+\frac{1}{m-2}}} \qquad(p_{\eps}\leq\theta).
\label{eq:p''}
\end{equation}
for any $\eps>0$. As a consequence of \eqref{eq:p'}-\eqref{eq:p''} we obtain the following asymptotes in the linear zone when $\eps\rightarrow 0^+$:
\begin{subequations}
\begin{align}
p_{\eps}(x)\sim &\;(m-2)c_0x\label{eq:asymptotp},\\
p_{\eps}'(x)\sim &\;(m-2)c_0\label{eq:asymptotp'},\\
p_{\eps}''(x)\sim &\;\frac{A}{\eps}\left(\frac{\eps}{x}\right)^{1+\frac{1}{m-2}}\label{eq:asymptotp''},
\end{align}
\label{eq:asymptotvarx}
\end{subequations}
with $A:=\frac{c_0}{[(m-2)c_0]^{\frac{1}{m-2}}}>0$. These only hold far enough from the boundary layer $x\gg\eps$ and as long as $p_{\eps}\leq\theta$ (otherwise the reaction term must be taken into accounted).
\par
Finally, asymptotic analysis of \eqref{eq:EDOp} at $x=-\infty$ (where $p=\eps$ and $p'=0$) easily yields for fixed $\eps>0$ the exponential decay
\begin{equation}
p_{\eps}'(x)\underset{x\rightarrow -\infty}{\sim} e^{c_{\eps}x/\eps},
\label{eq:equivp'-infty}
\end{equation}
and $|p_{\eps}(x)-\eps|\underset{-\infty}{=}\mathcal{O}\left(e^{\frac{c_{\eps}}{\eps}x}\right)$. This will be useful in the next section when building the maximal decay solution in the cold zone. Note again how the rate blowup $c_{\eps}/\eps\sim c_0/\eps\to +\infty$ translates the degeneracy into free-boundary.
%
%
\section{Cold zone and asymptotic expansion}
\label{section:zonefroide}
%
%
In this section we construct the solution $u$ of \eqref{eq:Elinx} with maximal decay in the cold zone in the form of an asymptotic expansion $u=u_0+\eps u_1+\eps u_2$, with $u_2\ll u_1$ in some sense. For the sake of simplicity we omit the subscripts and write here $c=c_{\eps}$ and $p=p_{\eps}$. Let us anticipate that perturbations will have a spatial structure resembling the one of the wave, including a boundary layer of thickness $\mathcal{O}(\eps)$: we therefore scale again as
$$
x=\xi\eps,\quad q(\xi)=\frac{1}{\eps}p(\eps\xi),\quad v(\xi)=u(\eps\xi).
$$
The slope of the reference wave solution jumps inside the boundary layer from $dp/dx=dq/d\xi\approx 0$ to $dp/dx=dq/d\xi\approx(m-2)c>0$, and this transition is steeper and steeper when $\eps\rightarrow 0^+$: we expect a singularity of the second derivative somewhere, which is of course consistent with the slope discontinuity in the asymptotic profile $p_0(x)$. In order to later neglect $p''(x)$ in the linear zone we will have to push the exit of the boundary layer far enough so that our asymptotic expansion encompasses this singularity. Asymptote \eqref{eq:asymptotp''} shows that in order to do so we need $x\gg \eps^{\frac{1}{m-1}}$: setting
\begin{equation}
 x_{\eps}:=\eps^{1-a},\qquad \xi_{\eps}:=x_{\eps}/\eps,\qquad a:=\frac{m-2}{m-1}\left(1+\frac{\eta}{2}\right)
\label{eq:def_xeps}
\end{equation}
for some $\eta>0$, it is easy to see that $ \eps^{\frac{1}{m-2}}\ll x_{\eps}\ll 1$ holds if $\eta>0$. As a consequence we may safely neglect $p''(x)$ for $x\gg x_{\eps}$. Here $\eta>0$ is exactly the one in Theorem~\ref{theo:reldisp} and arbitrarily small, but its occurrence in the above definition of $a$ is purely technical. In agreement with Theorem~\ref{theo:reldisp} we anticipate that the relevant values of $s$ should be of order $k^{1-\frac{1}{m-1}}\gg 1$: through this whole section one should think of $s$ as of $\gamma k^{1-\frac{1}{m-1}}$ for some fixed constant $\gamma$ of order unity and independent of $\eps,k$, which will be adjusted later. We also recall that we are interested in $(\eps,k)\to(0,+\infty)$ in the double limit \eqref{eq:regime_eta}.
\par
Since we set $x_{\eps}\gg \eps$ out of the boundary layer we have $p(x_{\eps})\underset{\eps\to 0}{\sim} (m-2)c_0 x_{\eps}\ll 1$, and we may therefore omit the reaction term (of ignition type) on the cold interval $I=]-\infty,x_{\eps}]$. Recasting \eqref{eq:Elinx} in $\xi$ coordinates as
\begin{equation}
\begin{array}{l}
Lv  =  \eps h v,\\
L  :=  -q\frac{d^2}{d\xi^2}+\left(c-\frac{2q'}{m-2}\right)\frac{d}{d\xi}-q'',\\
h =\left(s-k^2\eps^{\frac{m}{m-2}}q^{\frac{m}{m-2}}\right),
\end{array}
\label{eq:defL}
\end{equation}
we shall seek below solutions in the form $v=v_0+\eps v_1 +\eps v_2$ with $v_2\ll v_1$ in some sense. We will refer to $L$ as the principal operator, which is also the linearized operator for non-wrinkled perturbations $k=0$.

\subsection{Maximal decay and principal operator}
Let us recall that we investigate \textit{maximal decay} perturbations $u(-\infty)=0$: the asymptotic equation associated with \eqref{eq:Elinx} for $x=-\infty$ yields two characteristic exponential rates
\begin{equation}
r^{\pm}=\frac{c\pm\sqrt{c^2+4\eps(k^2\eps^{\frac{m}{m-2}}-s)}}{2\eps}.
\label{eq:r+-}
\end{equation}
For $s$ of order $k^{1-\frac{1}{m-1}}$ regime \eqref{eq:regime_eta} implies that $0<\eps(s-k^2\eps^{\frac{m}{m-2}})\ll 1$, and \eqref{eq:r+-} consequently yields
\begin{equation}
0<r^-\ll r^+\sim \frac{c}{\eps},\qquad r^+<\frac{c}{\eps}.
\label{eq:expocaracteristiquemaxzonefroide}
\end{equation}
Maximal decay obviously corresponds here to $u(x)\underset{-\infty}{=}\mathcal{O}\left(e^{r^+x}\right)$.
\par
In order to work in some fixed functional setting (independent of $\eps$) let us define the following weighted spaces on $I=]-\infty,\xi_{\eps}]$
$$
w(\xi)=\left\{\begin{array}{cc}
e^{-\frac{c_0}{2}\xi} & \xi\leq 0\\
1 & \xi>0
\end{array}\right. \quad \text{and}\quad 
\begin{array}{lcl}
B_{w}^0 & = & \left\{f\in\mathcal{C}_b, \quad fw\in\mathcal{C}_b(I)\right\},\\
B_{w}^k & = & \left\{f\in\mathcal{C}^k_b, \quad\forall j\leq k: f^{(j)}\in B_{w}^0\right\},\\
B_{w,0}^k & = & \left\{f\in B_w^k, \hspace{.2cm} f(\xi_{\eps})=0\right\}
\end{array}
$$
(recall that $c_{\eps}\to c_0>0$). Here $\mathcal{C}^k_b$ denotes the space of functions with continuous bounded first $k$ derivatives, and the $B_w$'s are equipped with their usual Banach norm. On one hand \eqref{eq:expocaracteristiquemaxzonefroide} shows that any maximal decay solution must behave as $e^{\eps r^+\xi}$ when $\xi\to -\infty$ with $\eps r^+\underset{\eps,k}{\sim} c_0>c_0/2$, and therefore belongs to $B_{w}^0$. On the other hand any non maximal decay solution behaves as $v(\xi)\underset{-\infty}{\sim} e^{\eps r^-\xi}$ with $0<\eps r^-\underset{\eps,k}{\ll} c_0/2$: such solutions cannot belong to $B_{w}^0$, and as a consequence it is legitimate to look for maximal decay solutions only in $B_{w}^0$.\label{page:discussion_Bw}
\begin{lem}
The principal operator $L:B_{w,0}^2\longrightarrow B_w^0$ is continuously invertible and
\begin{equation}
||L^{-1}||\leq C \eps^{-a}.
\label{eq:L-1}
\end{equation}
\label{lem:L-1}
\end{lem}
\begin{proof}
Defining $\Phi=:\int^{\xi}\frac{c}{q}$ it is easy to check that Duhamel's formula
\begin{equation}
f(\xi):=q'(\xi)\int\limits_{\xi}^{\xi_{\eps}}\left(\int\limits_{-\infty}^{z}\exp[\Phi(\eta)-\Phi(z)]\frac{g(\eta)}{q(\eta)q'(\eta)}d\eta\right)dz
\label{eq:duhamelLf=g}
\end{equation}
yields the unique solution in $B^2_{w,0}$ of $Lf=g$ for any given $g\in B_{w}$. The estimate for $L^{-1}$ then follows from \eqref{eq:asymptotvarx}-\eqref{eq:equivp'-infty} expressed in $\xi=x/\eps$ coordinates and the study of $\Phi$ for $\xi\to-\infty$ and $1\ll\xi\leq \xi_{\eps}=\eps^{-a}$ combined with \eqref{eq:duhamelLf=g}.
\end{proof}
%
%
\subsection{Asymptotic expansion and frequency regime}
\label{subsection_chap1:asymptotic_expansion}
Expanding $v=v_0+\eps v_1+\eps v_2$ with $v_2\ll v_1$ in some sense leads to solving $Lv=\eps h v$ as
\begin{subequations}
\begin{align}
Lv_0=&\;0,\label{eq:Lv0}\\
Lv_1=&\;hv_0,\label{eq:Lv1}\\
[L-\eps h]v_2=&\;\eps h v_1.\label{eq:Lv2}
\end{align}
\label{eq:Lvi=vj}
\end{subequations}
In order normalize perturbations $v(\xi_{\eps})=1$ we will require in addition that $v_0(\xi_{\eps})=1$ and $v_1(\xi_{\eps})=v_2(\xi_{\eps})=0$.
\begin{itemize}
\item
Resolution of \eqref{eq:Lv0}: the reference wave solution is as usual determined up to shifts and $L[q']=0$, thus yielding a suitable candidate $v_0=q'$ for the leading order. This is confirmed by \eqref{eq:equivp'-infty}, showing that $q'(\xi)=p'(\eps\xi)\underset{-\infty}{\sim} e^{c\xi}\underset{-\infty}{\ll} e^{\eps r^+\xi}$ decays fast enough. Normalizing finally leads to
$$
v_0:=\frac{q'}{q'(\xi_{\eps})}\in B_{w}^0.
$$
\item
Resolution of \eqref{eq:Lv1}: remark that $h=\left(s-k^2\eps^{\frac{m}{m-2}}q^{\frac{m}{m-2}}\right)\in L^{\infty}(I)$ implies $hv_0\in B_{w}^0$. Lemma~\ref{lem:L-1} then properly defines
$$
v_1:=L^{-1}[hv_0]\in B_{w,0}^2.
$$
\item
Resolution of \eqref{eq:Lv2}: taking advantage of the linear behaviour of $q(\xi)$ for large $\xi$ and regime \eqref{eq:regime_eta} with $s=\gamma k^{1-\frac{1}{m-1}}$ (for fixed $\gamma$ of order unity), we see that $h=s-k^2\eps^{\frac{m}{m-2}}q^{\frac{m}{m-2}}\sim s$ uniformly on $I=]-\infty,\xi_{\eps}]$: we may therefore consider the restriction of $L^{-1}$ to. The (continuous) restriction of $L^{-1}$ to the subspace $B_{w,0}^2$ therefore satisfies
\begin{equation}
||\eps L^{-1}[h\cdot]||_{\mathcal{L}(B_{w,0}^2)}\leq\eps||L^{-1}||.||h||_{\infty}\leq C \eps^{1-a}s\underset{\eps,k}{\ll} 1
\label{eq:epsL-1h<<1}
\end{equation}
according to Lemma~\ref{lem:L-1}, and $M:=\mathrm{Id}-\eps L^{-1}h\in\mathcal{L}(B_{w,0}^2)$ is hence close to Identity thus invertible. This finally allows to solve \eqref{eq:Lv2} as
$$
v_2:=M^{-1}(L^{-1}(\eps h v_1))\in B_{w,0}^2.
$$
\end{itemize}
\par
Regime \eqref{eq:regime_eta} also ensures that $v=v_0+\eps v_1+\eps v_2$ is really an asymptotic expansion in $B_w^2$, in the sense that $||\eps v_2||_{B_{w,0}^2}\ll||\eps v_1||_{B_{w,0}^2}\ll||v_0||_{B_{w}^0}=\mathcal{O}(1)$. More precisely, one easily shows using \eqref{eq:epsL-1h<<1} and $M\approx Id$ that
\begin{equation}
\begin{array}{ccccl}
v_1=L^{-1}(hv_0) & \Rightarrow  & ||\eps v_1||_{B_{w,2}^0}\leq C\eps^{1-a}s||v_0||_{B_{w}^0} & \underset{\eps,k}{\ll} & ||v_0||_{B_w^2},\\
v_2\approx L^{-1}(\eps h v_1) & \Rightarrow & ||\eps v_2||_{B_{w,2}^0}\leq C \eps^{1-a}s||\eps v_1||_{B_{w}^0} & \underset{\eps,k}{\ll} & ||\eps v_1||_{B_{w,0}^2}.
\end{array}
\label{eq:estimationepsv1epsv2}
\end{equation}
%
\subsection{Exit boundary conditions}
\label{subsection:conditionssortie}
We compute in this section  $(v,v')$ at the exit $\xi=\xi_{\eps}$ of the cold zone. Since we normalized $v(\xi_{\eps})=1$ it is enough to compute $v'$, and we estimate separately $v_0'(\xi_{\eps})$, $v_1'(\xi_{\eps})$ and $v_2'(\xi_{\eps})$.
\begin{itemize}
 \item 
Since we set the exit far enough out of the boundary layer, asymptotes \eqref{eq:asymptotvarx} hold at $x_{\eps}=\eps^{1-a}\gg\eps$ and therefore $\eps\frac{d^2 p}{dx^2}(x_{\eps})\underset{\eps\rightarrow 0}{\sim} A\eps^{a\left(1+\frac{1}{m-2}\right)}$. In terms of $\xi=x/\eps$ this corresponds to
\begin{equation}
v_0'(\xi_{\eps})=\mathcal{O}\left(\eps^{a\left(1+\frac{1}{m-2}\right)}\right).
\label{eq:v0'}
\end{equation}
\item
Let us recall that $Lv_1=hv_0=h\frac{q'}{q'(\xi_{\eps})}$ and $h\overset{L^{\infty}(I)}{\sim} s$, so that $v_1$ is close in $B_{w,0}^2$ to the solution of $Lv=s\frac{q'}{q'(\xi_{\eps})}$. Using \eqref{eq:duhamelLf=g} it is easy to compute explicitly $L^{-1}[q'](\xi)=(\xi_{\eps}-\xi)q'(\xi)$ hence
\begin{equation}
v_1(\xi)\approx \frac{s(\xi_{\eps}-\xi)q'(\xi)}{cq'(\xi_{\eps})} \quad\text{ in } B_{w,0}^2.
\label{eq:v1equiv}
\end{equation}
Since the $B_{w,0}^2$ topology controls the first two derivatives we obtain
\begin{equation}
v'_1(\xi_{\eps})\underset{\eps,k}{\sim}\frac{s}{c_{\eps}q'(\xi_{\eps})}\frac{d}{d\xi}\Big{[}(\xi_{\eps}-\xi)q'(\xi)\Big{]}_{\xi=\xi_{\eps}}=-\frac{s}{c_{\eps}}.
\label{eq:v1'}
\end{equation}
\item
Recalling that $M=\big{(}\mathrm{Id}-\eps L^{-1}h\big{)}\approx\mathrm{Id}$ in $\mathcal{L}(B_{w,0}^2)$, \eqref{eq:Lv2} and \eqref{eq:v1equiv} show that
$$
v_2  =  \big{(}\underbrace{\mathrm{Id}-\eps L^{-1}h}_{M\approx Id}\big{)}^{-1}L^{-1}[\underbrace{\eps h v_1}_{\approx \eps s v_1}]\approx \frac{\eps s^2}{c_{\eps}q'(\xi_{\eps})} L^{-1}\left[(\xi_{\eps}-\xi)q'\right]
$$
in $B_{w,0}^2$. Duhamel Formula \eqref{eq:duhamelLf=g} gives an explicit integral formula for $ L^{-1}\left[(\xi_{\eps}-\xi)q'\right]$: 
differentiating with respect to $\xi$, evaluating at $\xi=\xi_{\eps}$ and taking advantage of the asymptotic behaviour of $\Phi$ (for $\xi\to-\infty$ and $1\ll \xi\leq\xi_{\eps}$), one finally estimates
\begin{equation}
v'_2(\xi_{\eps})=\mathcal{O}\left(\eps^{1-a}s^2\right).
\label{eq:v2'}
\end{equation}
\end{itemize}
We claim now that the exit slope $v'(\xi_{\eps})$ is dominated by $\eps v_1'(\xi_{\eps})$ in the double limit $(\eps,k)\to(0,+\infty)$:
\begin{itemize}
\item 
From \eqref{eq:v0'} and \eqref{eq:v1'} we see that $|v_0'(\xi_{\eps})|\ll|\eps v_1'(\xi_{\eps})|$ holds as soon as $\eps^{a\left(1+\frac{1}{m-1}\right)}\ll \eps s$, which is true with $s=\gamma k^{1-\frac{1}{m-1}}$, $\eps\ll 1\ll k$ and by definition of $a=\frac{m-2}{m-1}\left(1+\frac{\eta}{2}\right)$.
\item
Similarly from \eqref{eq:v1'} and \eqref{eq:v2'} we see that $|\eps v'_2(\xi_{\eps})|\ll |\eps v'_1(\xi_{\eps})|$ holds in the regime \eqref{eq:regime_eta} with our choice of $a$.
\end{itemize}
Let us recall that $s$ is here of order $k^{1-\frac{1}{m-1}}\gg 1$ as anticipated from Theorem~\ref{theo:reldisp}: the new parameter
$$
\s:=\frac{s}{(m-2)c_{\eps}k^{1-\frac{1}{m-1}}}
$$
should therefore take values of order unity, and the exit conditions are finally summarized in terms of this new parameter $\s$ by
\begin{equation}
\left\{
\begin{array}{rcccl}
v(\xi_{\eps}) & = & v_0(\xi_{\eps}) & = & 1\\
\frac{dv}{d\xi}(\xi_{\eps}) & \sim & \eps \frac{dv_1}{d\xi}(\xi_{\eps}) & \sim & -\frac{\eps s}{c_{\eps}}=-(m-2)\s \eps k^{1-\frac{1}{m-1}}
\end{array}\right. .
\label{eq:sortiezonefroide}
\end{equation}
For technical reasons we will also need
\begin{prop}
For fixed $(\eps,k)$ the quantity $v'(\xi_{\eps})$ is continuously differentiable with respect to $\s$ and
\begin{equation}
 \frac{\partial}{\partial \s}\left[v'(\xi_{\eps})\right]=-(m-2)\eps k^{1-\frac{1}{m-1}}+r(\eps,k,\s)
\label{eq:dv'/dsigma}
\end{equation}
holds with $r(\eps,k,\s)=o\left(\eps k^{1-\frac{1}{m-1}}\right)$ when $(\eps,k)\overset{\mathcal{U}}{\rightarrow} (0,+\infty)$ and locally uniformly in $\s$.
\end{prop}
Remark that \eqref{eq:dv'/dsigma} is indeed consistent with a formal differentiation of \eqref{eq:sortiezonefroide} with respect to $\s$.
\label{prop:sortieZFd/dsigma}
\begin{proof}
Regularity with respect to $\s$ is a classical consequence of the linear dependence on $\s$ of \eqref{eq:defL}. Writing $z=\frac{\partial v}{\partial \s}$, estimate \eqref{eq:dv'/dsigma} is just a computation for $\frac{dz}{d\xi}(\xi_{\eps})=-(m-2)\eps k^{1-\frac{1}{m-1}}+...$ with $z=z_0+\eps z_1+\eps z_2$. Exactly as for the slope $\frac{dv}{d\xi}$ the order one determines here the dependence on $\s$ at the exit, namely $\frac{dz}{d\xi}(\xi_{\eps})\underset{\eps,k}{\sim}\eps \frac{dz_1}{d\xi}(\xi_{\eps})$. The technical computations are omitted here in order to keep this paper in a reasonable length. \end{proof}

\section{Hot zone and asymptotic problem}
\label{section:zonechaude}
In this section we take a formal limit $\eps=0$, $k=+\infty$ in some scaled $\z$ coordinates corresponding to a suitable (infinite) zoom in the linear zone. We obtain an asymptotic eigenvalue problem on the half line $\z>0$ ($\z=0$ corresponds to the exit of the boundary layer after zooming out), and this will relate the asymptotic coefficient $\gamma_0$ in Theorem~\ref{theo:reldisp} to some principal eigenvalue $\s_0$ of the limiting problem. The main issue in this section is precisely to compute this $\s_0$.
\\
\par
We anticipated that $s$ should be of order $k^{1-\frac{1}{m-1}}$, which also turns to be the suitable length-scale to investigate the linear region: scaling \eqref{eq:Elinx} as
\begin{equation}
\begin{array}{ccc}
\z=k^{1-\frac{1}{m-1}}x & \quad & \sigma=\frac{s}{(m-2)ck^{1-\frac{1}{m-1}}}\\
v(\z)=u\left(\frac{\z}{k^{1-\frac{1}{m-1}}}\right) & & q(\z)=k^{1-\frac{1}{m-1}}p\left(\frac{\z}{k^{1-\frac{1}{m-1}}}\right)
\end{array}
\label{def:variablezeta}
\end{equation}
for $k<\infty$ leads to
\begin{equation}
-qv''+\left(c-\frac{2q'}{m-2}\right)v'-q''v=\left((m-2)c\s-q^{\frac{m}{m-2}}+\frac{G'(q/k^{1-\frac{1}{m-1}})}{k^{1-\frac{1}{m-1}}}\right)v
\label{eq:Elinzetak} 
\end{equation}
with $q=q_{\eps,k}(\z)$ and where $p=p_{\eps}$ and $c=c_{\eps}$ depend only on $\eps$. This scaling is again  Lipschitz $\frac{dq}{d\z}=\frac{dp}{dx}$, and the parameter $\s$ takes values of order unity.
\par
We recall from Section~\ref{section:ondeplane} that in the limit $\eps\to 0$ we have $c_{\eps}\to c_0>0$, and the asymptotic wave profile $p_0=\displaystyle{\lim_{\eps\to 0}\, p_{\eps}}$ is exactly linear for $x\in[0,\frac{\theta}{(m-2)c_0}]$. In $\z$ coordinates 
the exit of the linear zone $x=x_{\theta}$ corresponds to
$$
\z_k:=\frac{\theta k^{1-\frac{1}{m-1}}}{(m-2)c_0}\underset{k}{\longrightarrow} +\infty,
$$
while the exit of the boundary layer $x=x_{\eps}$ is now
$$
\z_{\eps}=k^{1-\frac{1}{m-1}}x_{\eps}=k^{1-\frac{1}{m-1}}\eps^{1-a}\underset{\eps,k}{\longrightarrow} 0
$$
in the regime \eqref{eq:regime_eta}. The wave profiles behave as $q(\z)\approx q_0(\z)=(m-2)c_0\z$ in the linear zone $\z\in[\z_{\eps},\z_k]$, which grows to $]0,+\infty[$. Moreover since $\frac{dG}{dp}(p_0)$ and $\frac{d^2 p_0}{dx^2}$ are of order 1 for $x\in]0,+\infty[$, we may neglect
$$
q''(\z)\approx\frac{p_0''(x)}{k^{1-\frac{1}{m-1}}}\underset{k}{\rightarrow} 0,\qquad \frac{G'(q(\xi)/k^{1-\frac{1}{m-1}})}{k^{1-\frac{1}{m-1}}}\approx\frac{G'(p_0(x))}{k^{1-\frac{1}{m-1}}}\underset{k}{\rightarrow} 0
$$
at least formally when $\eps=0$ and $k=+\infty$ on $\z\in]0,+\infty[$. As a matter of fact a singularity of $\frac{d^2p_{\eps}}{dx^2}$ appears inside the boundary layer due to the slope discontinuity of $p_0$, but this will be carefully examined in Section~\ref{section:raccord} (this is precisely the reason why we pushed the exit of the cold $x=x_{\eps}$ zone far enough).
\par
Formally taking $\eps=0$, $k=+\infty$, $q(\z)=q_0(\z)=(m-2)c_0\z$, $\z_{\eps}=0$ and $\z_k=+\infty$ in \eqref{eq:Elinzetak} we obtain the following asymptotic problem:
\begin{equation}
\z\in]0,+\infty[,\qquad -\z v''-\frac{v'}{m-2}+bv=\s v,
\label{eq:Elinzeta}
\end{equation}
with
$$
b(\z)=B\z^{\frac{m}{m-2}},\qquad B:=[(m-2)c_0]^{\frac{2}{m-2}}>0.
$$
Since we were originally interested in perturbations $u(x)$ vanishing at infinity we also require in \eqref{eq:Elinzeta} that $v(+\infty)=0$, which is again a formal but suitable condition.
\begin{rmk}
For fixed $\s$ and $\z=+\infty$ equation \eqref{eq:Elinzeta} reads $-v''+B\z^{\frac{2}{m-2}}v=0$. One easily sees that there exists a one dimensional space of stable solutions $v(+\infty)=0$, while any other solution must blow $v(+\infty)=+\infty$ (both at least exponentially).
\label{rmk:eqasymptotiquezeta}
\end{rmk}
Recalling now that we also look for signed perturbations (as principal eigenfunctions), the parameter $\s$ clearly appears here as a principal eigenvalue to be determined for the problem \eqref{eq:Elinzeta} with associated zero boundary condition at infinity. Two main difficulties arise here compared to usual Sturm-Liouville theory: unboundedness $\z\in\R^+$ and singularity at $\z=0$. As a matter of fact we shall seek solutions which are $\mathcal{C}^1$ at $\z=0$, and one should think of this regularity as a boundary condition.
\\
\par
Throughout the rest of Section~\ref{section:zonechaude} we will consider $\eps=0$, $k=+\infty$, and the only relevant dependence will be the one on $\s$. In order to keep our notations as light as possible we will write $p=p_0(x)=(m-2)c_0x$, $q=q_0(\z)=(m-2)c_0\z$ and $c=c_0$. The main result in this section is:
\begin{theo}
There exists a principal eigenvalue $\s_0>0$ such that problem \eqref{eq:Elinzeta} has a \textit{positive} solution $v_0\in\mathcal{C}^1([0,+\infty[)\cap\mathcal{C}^2(]0,+\infty[)$ satisfying boundary conditions $v_0(0)=1$ and $v(+\infty)=0$. This principal eigenfunction satisfies in addition $v_0'(0)=-(m-2)\s_0$ and $v_0'< 0$.
\label{theo:vpprincipale}
\end{theo}
In this statement the $\mathcal{C}^1$ regularity at $\z=0$ and $v(+\infty)=0$ should be seen as a boundary condition, in contrast with $v_0(0)=1$ which is simply a normalization. Note that according to scaling \eqref{def:variablezeta} $\s_0$ of order unity means $s$ of order $k^{1-\frac{1}{m-1}}$: this will yield the asymptotic coefficient $\gamma_0=(m-2)c_0\s_0$ in Theorem~\ref{theo:reldisp}.
%
%
%
\subsection{The asymptotic principal eigenvalue}
In this section we first establish some technical results and then prove Theorem \ref{theo:vpprincipale}. We start by studying the singularity at $\z=0$.
\begin{prop}
For fixed $\s\in\R$ there exists a unique solution $v_{\s}\in\mathcal{C}^2(]0,+\infty[)\cap\mathcal{C}^1([0,+\infty[)$ of \eqref{eq:Elinzeta} such that $v_{\s}(0)=1$. In addition this solution satisfies $v'_{\s}(0)=-(m-2)\s$, and the mapping $\s\mapsto v_{\s}(.)$ is $\mathcal{C}^1$ from $\R $ into $\mathcal{C}^1([0,\z_0])$ for any fixed $\z_0>0$ .
\label{prop:vreguliereens}
\end{prop}
\begin{proof}
Remarking that the exponent in the zero-th order coefficient $b(\z)=B\z^{\frac{m}{m-2}}$ is $\frac{m}{m-2}>1$, solutions of \eqref{eq:Elinzeta} ``should not see'' this coefficient in the neighbourhood of $\z=0$ and therefore behave as
\begin{equation}
\z w''+\frac{w'}{m-2}+\s w=0.
\label{eq:Elinzetab=0}
\end{equation}
It is easy to obtain a regular solution $w\in\mathcal{C}^2(]0,+\infty[)\cap\mathcal{C}^1([0,+\infty[)$ of \eqref{eq:Elinzetab=0}  satisfying $w(0)=1$ and $w'(0)=-(m-2)\s$ in the form of a power series. This solution is unique (see \cite{CodLev-ODE} Theorem 6.1 p.169, first kind singularities) and classically regular with respect to $\s$.
\par
We may now look for solutions of \eqref{eq:Elinzeta} in the form $v=w+h$: $h$ should obviously solve a non-homogeneous equation involving $w$, and the limit $\z\rightarrow 0$ also yields $h(0)=h'(0)=0$ (remember that $w(0)=1$ and $w'(0)=-(m-2)\s$). The proof consists in two steps: a classical Banach fixed point on the integral formulation for $h$ first yields a solution on $[0,\o{\z}]$ for fixed small $\o{\z}>0$, thus stepping away from the singularity. This solution then extends to the right $[\o{\z},+\infty[$, where the ODE is now regular. Uniqueness is obtained as for $w$ since the singularity at $\z=0$ is of the first kind, see again \cite{CodLev-ODE}. Regularity of $h$ with respect to $\s$ is obtained first on $[0,\o{\z}]$ applying the Implicit Functions Theorem in the integral formulation, and then on $[\o{\z},\z_0]$ by classical regularity of Cauchy solutions with respect to initial conditions and parameters. \end{proof}
For the sake of clarity we will denote by $v=v_{\s}$ this unique regular solution of \eqref{eq:Elinzeta} and will only consider $\s\geq 0$.
\begin{prop}
For $\s\geq 0$ only three scenarios are possible for \eqref{eq:Elinzeta} at infinity: $v(+\infty)=+\infty$, $v(+\infty)=0$ and $v(+\infty)=-\infty$. Moreover, if $\z_{\s}=\left(\frac{\s}{B}\right)^{\frac{m-2}{m}}$ is the first time where $b(\z)=\s$, the following holds:
\begin{enumerate}
 \item 
If there exists $\z_1\geq\z_{\s}$ such that $v(\z_1)>0$ and $v'(\z_1)\geq0$, then $v(\z)\geq v(\z_1)>0$ on $[\z_1,+\infty[$ and $v(+\infty)=+\infty$.
\item
If there exists $\z_1\geq\z_{\s}$ such that $v(\z_1)<0$ and $v'(\z_1)\leq0$, then $v(\z)\leq v(\z_1)<0$ on $[\z_1,+\infty[$ and $v(+\infty)=-\infty$.
\end{enumerate}
\label{prop:alternative+infty}
\end{prop}
\begin{proof}
When $\z=+\infty$ \eqref{eq:Elinzeta} reads $-v''+B\z^{\frac{2}{m-2}}v=0$ and therefore satisfies the classical Maximum Principle ($B=cst>0$). As a consequence either $v(+\infty)=\pm\infty$, either $v(+\infty)=cst$, and the only possible finite limit is clearly $v(+\infty)=0$. Recasting \eqref{eq:Elinzeta} as $[\z^{\frac{1}{m-2}}v']'=\frac{b-\s}{\z^{1-\frac{1}{m-2}}}v$ we have by definition $b(\z)-\s>0$ for $\z>\z_{\s}$, and the rest of the statement is easily obtained integrating from $\z_1$ to $+\infty$. \end{proof}
\begin{prop}
For $\s\geq 0$ small enough $v>0$ holds on $[0,+\infty[$, and $v(+\infty)=+\infty$.
\label{prop:sigmapetit}
\end{prop}
\begin{proof}
Denoting by $\o{v}$ the solution for $\s=0$, \eqref{eq:Elinzeta} reads $[\z^{\frac{1}{m-2}}\o{v}']'=B\z^{\frac{3}{m-2}}\o{v}$. Using boundary conditions $\o{v}(0)=1$ and $\o{v}'(0)=0$ it is easy to integrate from $\z=0$ to $\z>0$ and see that $\o{v}(+\infty)=+\infty$. This easily extends by continuity to $\s>0$ small enough remarking that $\z_{\s}:=\left(\frac{\s}{B}\right)^{\frac{m-2}{m}}\rightarrow 0^+$ when $\s\rightarrow 0^+$ and using the first case in Proposition~\ref{prop:alternative+infty}.
\end{proof}
\begin{prop}
For $\s$ large enough there exists a time $\z_-=\z_-(\s)>0$ such that $v(\z_-)<0$.
\label{prop:v<0}
\end{prop}
\begin{proof}
Scaling \eqref{eq:Elinzeta} as $t=\s\z$ and $y(t)=v\left(\frac{t}{\s}\right)$ yields
$$
t\ddot{y}+\alpha\dot{y}+y=\beta t^{\frac{m}{m-2}}y, \hspace{.3cm}y(0)=1,\hspace{.3cm}\dot{y}(0)=-\frac{1}{\alpha}
$$
with $\dot{y}=\frac{dy}{dt}$, $\alpha=\frac{1}{m-2}\in]0,1[$ and $\beta=\frac{B}{\s^{2\frac{m-1}{m-2}}}$. When $\beta=0$ we may look for the solution $y_{\alpha}(t)$ in the form of a power series, and an easy but technical computation shows that this power series takes negative values $y_{\alpha}(t=2\alpha)<0$ for any fixed $\alpha\in]0,1[$. This easily extends by continuity to $\beta>0$ small, i-e to $\s$ large.
\end{proof}
We may now prove our main statement:
\begin{proof}(of Theorem \ref{theo:vpprincipale}). By Propositions \ref{prop:sigmapetit} and \ref{prop:v<0} the quantity
\begin{equation}
\s_{0}=:\sup_{\s\geq 0}\Big{(}\s\geq 0,\hspace{.5cm} \s'\in[0,\s]\Rightarrow v_{\s'}(.)>0\Big{)}
\label{eq:defsigma0}
\end{equation}
is finite and positive. We claim that $\s_0$ is indeed the desired principal eigenvalue: denoting by $v_0$ the corresponding solution for $\s=\s_0$ we have by continuity that $v_0=\displaystyle{\lim_{\s\to\s_0^-}\,v_{\s}}\geq 0$, and actually $v_0>0$ (unless $v_0\equiv 0$, which contradicts $v_0(0)=1$). Stability $v_0(+\infty)=0$ is a consequence of Proposition~\ref{prop:alternative+infty}: the scenario $v_0(+\infty)=+\infty$ is impossible since otherwise there would exist a right neighbourhood of $\s_0$ in which $v_{\s}>0$, thus contradicting the definition of $\s_0$. In order to obtain monotonicity, let $\z_0:=\left(\frac{\s_0}{B}\right)^{\frac{m-2}{m}}$: by Proposition~\ref{prop:alternative+infty} $v_0>0$ must satisfy $v_0'<0$ on $[\z_0,+\infty[$ (otherwise $v_0(+\infty)=+\infty$, which we proved to be impossible). On $[0,\z_0]$ monotonicity is a consequence of $v_0'(0)=-(m-2)\s_0<0$ and $Lv_0=(\s_0-b)v_0\leq 0$, where $L=-\z\frac{d^2}{d\z^2}-\frac{1}{m-2}\frac{d}{d\z}$ is elliptic and has no zero-th order coefficient.
 \end{proof}
%
%
\subsection{Isolated eigenvalue}
We prove here that there are no other eigenvalues close to $\s_0$. More precisely we show that $\s\in[0,\s_0[\Rightarrow v(+\infty)=+\infty$ and that for $\s>\s_0$ close enough $v(+\infty)=-\infty$ holds. The key point is the following monotonicity:
\begin{prop}
For $\s\in[0,\s_0]$, the mapping $\s\mapsto v_{\s}(.)$ is strictly pointwise decreasing on $]0,+\infty[$.
\label{prop:vdecroissantes}
\end{prop}
\begin{proof}
Remark that $\z=0$ is irrelevant since we normalized $v_{\s}(0)=1$ for all $\s$. The proof is quite standard: if $\s_1>\s_2\in[0,\s_0]$ then $v_1,v_2>0$, $\alpha :=\frac{v_1}{v_2}>0$ is well defined and satisfies
$$
\tilde{L}[\alpha]:=-\z v_2 \alpha''-\left(2\z v_2'+\frac{v_2}{m-2}\right)\alpha'=(\s_1-\s_2)v_1>0.
$$
Computing $\alpha(0)=1$ and $\alpha'(0)=(m-2)(\s_2-\s_1)<0$ shows that $\alpha<1$ at least in the neighbourhood of $\z=0$. If $\alpha(\z^*)\geq 1$ for some $\z^*>0$ then $\alpha$ would attain a non-negative local minimum point $\z_m\in]0,\z^*[$, which would contradict the classical Minimum Principle. Thus $\alpha< 1\Leftrightarrow v_1< v_2$ on $]0,+\infty[$.
\end{proof}
\begin{prop}
For $\s\in[0,\s_0[$ we have that $v_{\s}(+\infty)=+\infty$.
\label{prop:v(+infty)=+infty}
\end{prop}
\begin{proof}
For such $\s$ we have by monotonicity $v_{\s}> v_{\s_0}>0$ hence by Proposition~\ref{prop:alternative+infty} either $v_{\s}(+\infty)=+\infty$, either $v_{\s}(+\infty)=0$. Assume by contradiction that $v_{\s}(+\infty)=0$ for some $\s\in]0,\s_0[$: both $v_{\s}$ and $v_{\s_0}$ then satisfy the same asymptotic equation when $\z\to+\infty$ and thus behave similarly. Slowly increasing $\alpha$ from zero, we see that $\alpha v_\s<v_{\s_0}$ holds until a first critical value $\alpha=\alpha^*\in]0,1[$ for which there exists a contact point $\z_m$ between $\alpha^*v_\s\leq v_{\s_0}$ and $v_{\s_0}$. On the other hand $z:=\alpha^*v_{\s}-v_{\s_0}\leq 0$ satisfies $Lz=\underbrace{(\s-\s_0)}_{\leq 0}\alpha^* v_{\s}+\s_0 \underbrace{z}_{\leq 0}\leq 0$ with $L=-\z\frac{d^2}{d\z^2}-\frac{1}{m-2}\frac{d}{d\z}+b(\z)$, thus contradicting the Minimum Principle.
\end{proof}
\begin{rmk}
In the proof above we only used the fact that $\s_0$ is associated with a (stable) positive eigenfunction $v_{\s_0}$. As a classical byproduct we retrieve uniqueness of the principal eigenvalue $\s_0$: by Proposition~\ref{prop:v(+infty)=+infty} there are no eigenvalues at all for $\s<\s_0$, and if there existed an other principal eigenvalue $\s_1>\s_0$ associated with a positive eigenfunction $v_{\s_1}$ we could repeat exactly the same argument and conclude that $v_{\s}>v_{\s_1}$ and $v_{\s}(+\infty)=+\infty$ for all $\s<\s_1$, which fails for $\s=\s_0<\s_1$.
\label{rmk:vp_unique}
\end{rmk}
\begin{lem}
If $\s>\s_0$ then $v_{\s}\geq 0$ cannot hold on $\R^+$.
\label{lem:sturm_negative}
\end{lem}
\begin{proof}
Assume by contradiction that $v_{\s}\geq 0$ for some $\s>\s_0$. Once again we must have either $v_{\s}(+\infty)=0$, either $v_{\s}(+\infty)=+\infty$. Owing to Remark~\ref{rmk:vp_unique} above the first case is impossible, and therefore $v_{\s}(+\infty)=+\infty$. Defining $\alpha:=\frac{v_{\s}}{v_{\s_0}}$ we have $\alpha\geq 0$, $\alpha'(0)=(m-2)(\s_0-\s)<0$ and $\alpha(+\infty)=+\infty$. As a consequence $\alpha$ attains a (non-negative) minimum point at some $\z_m>0$, thus contradicting the classical Minimum Principle and
$$
\tilde{L}[\alpha]:=-\z v_0 \alpha''-\left(2\z v_0'+\frac{v_0}{m-2}\right)\alpha'=(\s-\s_0)v_0>0.
$$ \end{proof}
\begin{prop}
For $\s>\s_0$ close enough there holds $v_{\s}(+\infty)=-\infty$.
\label{prop:v(+infty)=-infty}
\end{prop}
\begin{proof}
By Proposition~\ref{prop:alternative+infty} it is enough to show that for such $\s$ there exists a time $\z>\z_{\s}=\left(\frac{\s}{B}\right)^{\frac{m-2}{m}}$ such that $v_{\s}(\z)<0$ and $v_{\s}'(\z)<0$. For any $\s>\s_0$ we normalized $v_{\s}(0)=1>0$, and by Lemma~\ref{lem:sturm_negative} $v_{\s}(.)$ takes negative values at least somewhere in $\R^+$: the first time $\z_{0}(\s)>0$ where $v_{\s}$ vanishes is hence well defined. Since $v_{\s_0}>0$ we must have by continuity $\z_{0}(\s)\rightarrow +\infty$ when $\s\rightarrow \s_0^+$, and in particular $\z_{0}(\s)>\z_{\s}$ since $\z_{\s}\to \left(\frac{\s_0}{B}\right)^{\frac{m-2}{m}}<+\infty$. By definition of $\z_{0}$ we also have $v_{\s}>0$ on $[0,\z_{0}[$ and $v_{\s}(\z_0)=0$ hence $v_{\s}'(\z_{0})\leq 0$, and actually $v_{\s}'(\z_{0})< 0$. Thus the desired behaviour for times $\z\approx \z_0(\s)^+$ and $\s>\s_0$ close. \end{proof}
%
%
\subsection{Analyticity}
\label{subsection:analyticite}
For $\s>0$ equation \eqref{eq:Elinzeta} has a left branch of solutions $v_l(\s,.)$ which are $\mathcal{C}^1$ at $\z=0$ and normalized as $v_l(\s,0)=1$ (see Proposition~\ref{prop:vreguliereens}). According to Remark~\ref{rmk:eqasymptotiquezeta} there also exists for any $\s$ a right branch of stable solutions $v_r(\s,+\infty)=0$, which we normalize as $v_r(\s,\z_0)=1$ for some $\z_0>0$. In this section we show that these two branches are analytical in $\s$, which we will need later for technical reasons.
\begin{prop}
For any fixed $\z_0>0$ there exists $R(\z_0)>0$ such that $\s\mapsto v_l(\s,.)$ is analytical into $\mathcal{C}^1([0,\z_0])$ with radius of convergence at least $R(\z_0)$.
\label{prop:vganalytique}
\end{prop}
\begin{proof}
If no singularity occurred at $\z=0$ in \eqref{eq:Elinzeta} this would be a classical consequence of the linear dependence of the equation on $\s$. We construct below a solution $v\in\mathcal{C}^1([0,\z_0])$ of \eqref{eq:Elinzeta} which is analytical in $\s$ and such that $v(0)=1$. By uniqueness in Proposition~\ref{prop:vreguliereens} this solution will agree with $v_l(\s,.)$.
\par
For $\s\approx \s_0$ let us enlighten $\s-\s_0$ and rewrite \eqref{eq:Elinzeta} in the form
\begin{equation}
Lv:=-\z v''-\frac{v'}{m-2}+(b-\s_0)v=(\s-\s_0)v.
\label{eq:Lsigma-sigma0}
\end{equation}
Looking for solutions in power series $v(\z)=\sum_{n\geq 0}(\s-\s_0)^nv_n(\z)$ leads to
$$
\begin{array}{ccccl}
 n=0 & : & Lv_0 & = & 0\\
n\geq 1 & : & Lv_n & = & v_{n-1}.
\end{array}
$$
We naturally set $v_0:=v_{\s_0}$ to be the principal eigenfunction in Theorem \ref{theo:vpprincipale}, and also require $v_n$ to be $\mathcal{C}^1$ at $\z=0$ and satisfy $v_n(0)=0$ for $n\geq 1$ (so that $v(0)=1$ in the end). For any $g\in\mathcal{C}([0,\z_0])$ it is easy to check that Duhamel formula
$$
f(\z)= -v_0(\z)\displaystyle{\int\limits_0^{\z}\frac{1}{v_0^2(s)s^{\frac{1}{m-2}}}\left(\int\limits_0^{s}v_0(t)t^{\frac{1}{m-2}-1}g(t)\mathrm{d}t\right)\mathrm{d}s}.
$$
yields the unique solution $Lf=g$ satisfying the above boundary conditions, and that $||f||_{\mathcal{C}^1}\leq C||g||_{\infty}$ for some $C(\z_0)>0$. Equation $Lv_n=v_{n-1}$ is therefore uniquely solvable as $v_n=L^{-1}v_{n-1}$ and an induction argument immediately yields $||v_n||_{\mathcal{C}^1}\leq C^n||v_0||_{\infty}$, thus convergence in $\mathcal{C}^1$ with radius at least $R(\z_0):=\frac{1}{C(\z_0)}$.
\end{proof}
We have a similar result for the stable right branch:
\begin{prop}
For any fixed $\z_0>0$ there exists $R(\z_0)>0$ such that $\s\mapsto v_r(\s,.)$ is analytical into $\mathcal{C}^1_b([\z_0,+\infty[)$ with radius of convergence at least $R(\z_0)$.
\label{prop:vdanalytique}
\end{prop}

\begin{proof}
The argument is almost identical, the only difference being that we want now $(v_n)_{n\geq 1}$ to decay at infinity (whereas we had to handle the singularity at $\z=0$ for the left branch). The corresponding Duhamel Formula for $Lf=g$ reads now
$$
f(\z)  = v_0(\z) \displaystyle{\int\limits_{\z_0}^{\z}\frac{1}{v_0^2(s)s^{\frac{1}{m-2}}}\left(\int\limits_{s}^{+\infty}v_0(t)t^{\frac{1}{m-2}-1}g(t)\mathrm{d}t\right)\mathrm{d}s},
$$
and the rest of the proof is straightforward. \end{proof}

\subsection{Functional setting multiplicities}
\label{section:cadrefonctionnel}
Let us remind at this point equation \eqref{eq:Elinzeta}
$$
\z>0:\qquad-\z u''-\frac{u'}{m-2}+bu=\s u,
$$
which is an eigenvalue problem for $L:=-\z\frac{d^2}{d\z^2}-\frac{1}{m-2}\frac{d}{d\z}+b(\z)$. Boundary conditions are (i) $\mathcal{C}^1$ regularity at $\z=0$ and (ii) decay $v(+\infty)=0$. So far we did not define any clear functional setting, and this is the purpose of this section.
\par
Since $v(+\infty)=0$ the largest possible space on which $L$ could act is clearly a subspace of
$$
E:=\Big{\{}f\in\mathcal{C}([0,+\infty]),\quad f(+\infty)=0\Big{\}},
$$
and our functional setting should also take into account the singularity $\z=0$.
\begin{prop}
For any $f\in E$ there exists a unique solution $v\in E$ of $Lv=f$ having $\mathcal{C}^1$ regularity at $\z=0$. This solution can be computed as
\begin{equation}
 v(\z)=L^{-1}[f](\z):=\overline{v}(\z)\displaystyle{
\int\limits_{\z}^{+\infty}\frac{1}{\overline{v}^2(s)s^{\frac{1}{m-2}}}\left(
\int\limits_0^s \overline{v}(t)t^{\frac{1}{m-2}-1}f(t)\mathrm{d}t
\right)\mathrm{d}s
},
\label{eq:formulintv(zeta)}
\end{equation}
where $\overline{v}(\z)>0$ is the solution in Proposition~\ref{prop:vreguliereens} for $\s=0$. We have moreover $v'(0)=-(m-2)f(0)$.
\label{prop:domain_L}
\end{prop}
\begin{proof}
The $\mathcal{C}^1$ regularity at $\z=0$ is easy to retrieve, and one readily checks that $L^{-1}$ maps $E$ to $E$ (or in other words that $v$ decays at infinity if $f$ does).
\end{proof}
This defines very naturally the domain
$$
D(L):=L^{-1}(E)\subset E,
$$
where $L^{-1}$ is given by Proposition~\ref{prop:domain_L}. Any $u\in D(L)$ therefore satisfies both boundary conditions, and let us point out that if $f=L^{-1}g$ with $g\in E$ then $f$ is of course a classical $\mathcal{C}^2(]0,+\infty[)$ solution of the corresponding ODE.
\begin{prop}
$\s_0$ is an eigenvalue of $L:D(L)\subset E\rightarrow E$ with geometric multiplicity $m_g(\s_0)=1$.
\label{prop:multgeom=1}
\end{prop}
This is of course consistent with the one-dimensionality in Theorem~\ref{theo:reldisp}.
\begin{proof}
For $\s=\s_0$ the principal eigenfunction $v_0=v_{\s_0}$ constructed in Theorem \ref{theo:vpprincipale} solves the ODE, has the required regularity and decays at infinity, thus belongs to $D(L)$ and is indeed a proper eigenvalue. Regarding it's geometric multiplicity we recall from Remark~\ref{rmk:eqasymptotiquezeta} that for any $\s$ there always exists a solution of the ODE which blows (exponentially) at infinity and therefore cannot belong to $E$.
\end{proof}
The algebraic multiplicity will also be important in Section~\ref{section:raccord}:
\begin{prop}
$\s_0$ has finite algebraic multiplicity $m_a(\s_0)=1$.
\label{prop:multalg=1}
\end{prop}
\begin{proof}
If $v_0$ denotes again our principal eigenfunction and $L_0:=L-\s_0$ we need to show that $\mathrm{Ker}(L_0^2)=\mathrm{Ker}(L_0)$ is generated by $v_0$; in other words, that there is no solution $v\in D(L)$ of $Lv=v_0\Leftrightarrow -\z v''-\frac{v'}{m-2}+(b-\s_0)v=v_0$. Assuming by contradiction that $v$ is such a solution, Variation of Constants $v(\z)=\alpha(\z) v_0(\z)$ and boundary conditions lead to
$$
\alpha'(\z)=-\frac{1}{\z^{\frac{1}{m-2}}v_0^2(\z)}\displaystyle{\int\limits_0^{\z}s^{\frac{1}{m-2}-1}v_0^2(s)\,\mathrm{d}s}.
$$
Decay $v_0(+\infty)=0$ (at least exponentially) implies that the integral above is absolutely convergent at infinity, hence that $\alpha'(\z)\underset{+\infty}{\sim}-\frac{C}{v_0^2(\z)\z^{\frac{1}{m-2}}}\to -\infty$ with $C=\int_0^{+\infty}s^{\frac{1}{m-2}-1}v_0^2(s)\,\mathrm{d}s>0$ (let us also recall that $v_0>0$). Using the equation for $v_0$ it is finally easy to see that $\alpha'$ blows so fast at infinity that $v=\alpha v_0$ does too, thus contradicting $v(+\infty)=0$ and $v\in D(L)$.
\end{proof}
\begin{prop}
$\s=\s_0\in \R$ is an isolated eigenvalue in $\mathbb{C}$.
\label{prop:s0vpisolee}
\end{prop}
\begin{proof}
Rewriting $Lv=\s v\Leftrightarrow -\z v''-\frac{1}{m-2}v'+b(\z)v=\s v$ in the divergence form
$$
-\left(\z^{\frac{1}{m-2}}v'\right)'+\frac{b(\z)}{\z^{1-\frac{1}{m-2}}}v=\frac{\s}{\z^{1-\frac{1}{m-2}}}v,
$$
we see that the problem is self-adjoint and it is enough to prove that $\s_0$ is isolated in $\R$. Let us remind that we are looking for eigenfunctions $v\in D(L)$, hence $\mathcal{C}^1$ at $\z=0$ and decaying at infinity. As a result any eigenfunction $v_{\s}$ must agree with the unique regular solution of Proposition~\ref{prop:vreguliereens}, which according to Proposition~\ref{prop:v(+infty)=+infty} and Proposition~\ref{prop:v(+infty)=-infty} cannot decay if $\s$ is close to $\s_0$. \end{proof}
%
%
\section{Asymptotic matching}
\label{section:raccord}
In this section we match the physical setting $\eps>0,k<+\infty$ with the (formal) asymptotic problem $\eps=0$, $k=+\infty$ of Section~\ref{section:zonechaude}.
\\
\par
For $s$ of order $k^{1-\frac{1}{m-1}}$ we constructed in Section~\ref{section:zonefroide} the solution of \eqref{eq:Elinx} in $\xi=x/\eps$ coordinates with maximal decay at $-\infty$, normalized as $v(\xi_{\eps})=1$. This was done for $\eps>0$ and $k<+\infty$ (small and large enough) so we may indistinctly switch from $\xi$ to $x$ coordinates. This left solution extends to $x>x_{\eps}$ and will be denoted below by $u_l(x)=u_l(\eps,k,s,x)$.
\par
On the right side $x=+\infty$ we have $p=1$, $p'=p''=0$ and the characteristic equation associated with \eqref{eq:Elinx} reads $-r^2+c_{\eps}r+(k^2-s-G'(1))=0$, yielding two exponential rates $r^{\pm}$. In the regime \eqref{eq:regime_eta} with $s$ of order $k^{1-\frac{1}{m-1}}\ll k^2$ (and $c_{\eps}\to c_0>0$) these are
$$
r^{\pm}\underset{\eps,k}{\sim}\pm k,
$$
and there always exists a unique branch of stable solutions $u(+\infty)=0$ which we will denote by $u_r(x)=u_r(\eps,k,s,x)$.
\par
Instead of parameters $k\rightarrow +\infty$, $s=\mathcal{O}\left(k^{1-\frac{1}{m-1}}\right)\rightarrow +\infty$ and $x$ coordinates we will rather use
$$
\delta=\frac{1}{k^{1-\frac{1}{m-1}}}\rightarrow 0,\qquad \s=s\frac{\delta}{(m-2)c_{\eps}}=\mathcal{O}(1),\qquad\z=x/\delta,
$$
and \eqref{eq:Elinzetak} reads
\begin{equation}
\left\{
\begin{array}{l}
 Lv:=-qv''+\left(c-2\frac{q'}{m-2}\right)v'+\left(\left(q\right)^{\frac{m}{m-2}}-(m-2)c\s-q''-\delta G'(\delta q)\right)v=0,\\
c=c_{\eps}>0, \qquad q=q_{\eps,\delta}(\z)=\frac{1}{\delta}p_{\eps}(\delta \z).
\end{array}
\right.
\label{eq:Elinzetadelta}
\end{equation}
In terms of these new parameters $(\eps,\delta)$ the frequency regime \eqref{eq:regime_eta} becomes
\begin{equation}
 \omega:=\Big{\{}\eps,\delta>0:\qquad 0<\eps^{-\eta\frac{m-2}{m-1}+\frac{1}{m-1}}<\delta<\delta_0\Big{\}},
\label{eq:regimefreqdelta}
\end{equation}
where $\eta>0$ is again exactly the one in Theorem~\ref{theo:reldisp}, and we are now interested in the double limit $(\eps,\delta)\overset{\omega}{\to}(0,0)$.
\par
Throughout this entire section we will write $v_l(\eps,\delta,\s,\z)$ for the left branch coming from the cold zone and $v_{r}(\eps,\delta,\s,\z)$ for the stable right branch normalized as $v_{r}(\z_0)=1$ (for some $\z_0>0$ to be chosen later). Those are simply $u_l$ and $u_r$ expressed in $\z$ variables and in terms of $(\eps,\delta)$ instead of $(\eps,k)$. We will also denote by $v_{l0}(\s,\z)$ the regular solution of the formal asymptotic problem \eqref{eq:Elinzeta} (defined in Proposition~\ref{prop:vreguliereens}), and $v_{r0}(\s,\z)$ its stable right branch normalized as $v_{r0}(\s,\z_0)=1$ (see Remark \ref{rmk:eqasymptotiquezeta}). We will prove that
$$
\begin{array}{c}
v_l(\eps,\delta,\s,.)\rightarrow v_{l0}(\s,.)\\
\frac{\partial v_l}{\partial \s}(\eps,\delta,\s,.)\rightarrow \frac{\partial v_{l0}}{\partial \s}(\s,.)
\end{array}
\qquad \text{and}\qquad 
\begin{array}{c}
v_r(\eps,\delta,\s,.)\rightarrow v_{r0}(\s,.)\\
\frac{\partial v_r}{\partial \s}(\eps,\delta,\s,.)\rightarrow \frac{\partial v_{r0}}{\partial \s}(\s,.)
\end{array}
$$
when $(\eps,\delta)\rightarrow (0,0)$ in $\omega$. This should be no surprise since the asymptotic problem was precisely obtained in the (formal) limit $(\eps,k)=(0,+\infty)$ or equivalently $(\eps,\delta)=(0,0)$.
\\
\par
Our first step is to establish convergence of the scaled wave profile $q_{\eps,\delta}\rightarrow q_0$ in the linear zone $\z\in[\z_{\eps},\z_{\delta}]$, where $\z_{\eps}=x_{\eps}/\delta$ corresponds to the exit of the cold zone $x=x_{\eps}=\eps^{1-a}$ and $\z_{\delta}=x_{\theta}/\delta$ to the exit of the linear zone $x=x_{\theta}$ defined in \eqref{eq:defxthetatranslationpeps}. One can check that
$$
\z_{\eps}\to 0 \qquad \text{and}\qquad  \z_{\delta}\to+\infty
$$
when $(\eps,\delta)\overset{\omega}{\to}(0,0)$, and the linear zone stretches to $]0,+\infty[$ in this scaling.
\begin{prop}
We have convergence $||q_{\eps,\delta}(\z)-(m-2)c_0\z||_{\mathcal{C}^2([\z_{\eps},\z_{\delta}])}\rightarrow 0$ in the double limit \eqref{eq:regimefreqdelta}.
\label{prop:q->q0}
\end{prop}
The proof is technical but straightforward when taking advantage of Proposition~\ref{prop:peps->p0} and asymptotes \eqref{eq:asymptotvarx}. In order to retrieve the above $\mathcal{C}^2$ convergence it is important that the exit of the cold zone was pushed far enough out of the boundary layer, thus allowing to neglect $p_{\eps}''(x)$ in the linear zone.

%
%
\subsection{Regularity of the branches}
\label{section:branches_regularity}
We start by establishing convergence of the left branches as claimed above:
\begin{prop}
Fix $\z_0>0$ and $\s_*$: then $v_l(\eps,\delta,\s,.)\to v_{l0}(\s_*,.)$ in $\mathcal{C}^1([\z_{\eps},\z_0])$ when $(\eps,\delta,\s)\rightarrow (0,0,\s_*)$ in the double limit \eqref{eq:regimefreqdelta}.
\label{prop:vg->vg0}
\end{prop}
Note how we handle simultaneously the double limit $(\eps,\delta)\to(0,0)$ and continuity with respect to $\s$.
\begin{proof}
For $\eps,\delta>0$ the left branch $v_l=v_l(\eps,\delta,\s,\z)$ solves \eqref{eq:Elinzetadelta}:
$$
Lv_l=0,\quad L=-q\frac{d^2}{d\z^2}+\left(c-\frac{2q'}{m-2}\right)\frac{d}{d\z}+\left(q^{\frac{m}{m-2}}-(m-2)c\s-q''-\delta G'(\delta q)\right),
$$ 
while the asymptotic left branch $v_{l0}=v_{l0}(\s_*,\z)$ solves \eqref{eq:Elinzeta}:
$$
L_0v_{l0}=0,\quad L_0=-(m-2)c_0\z\frac{d^2}{d\z^2}-c_0\frac{d}{d\z}+(m-2)c_0(b(\z)-\s_*).
$$
Proposition~\ref{prop:q->q0} guarantees that the coefficients of $L$ are very close to those of $L_0$ when $(\eps,\delta,\s)\to(0,0,\s_*)$. Boundary conditions \eqref{eq:sortiezonefroide} moreover read in $\z$ coordinates
$$
v_{l}(\z_{\eps})=1,\qquad v_l'(\z_{\eps})\underset{\eps,\delta}{\sim} -(m-2)\s
$$
with $\z_{\eps}\rightarrow 0$. According to Proposition~\ref{prop:vreguliereens} the asymptotic left branch satisfies
$$
v_{l0}(0)=1,\qquad v_{l0}'(0)=-(m-2)\s_*,
$$
and we see that initial values are also very close. These two branches solve two very similar Cauchy problems, and should therefore stay very close on the interval $[\z_{\eps},\z_0]$. The difficulty is here that the asymptotic Cauchy problem degenerates at $\z=0$ when $\eps,\delta\rightarrow 0$, thus preventing from applying classical regularity of Cauchy solutions with respect to parameters and initial values. In order to keep this paper in a reasonable length we only sketch the argument below.
\par
The proof relies on a stability argument for $z:=v_l-v_{l0}$, which satisfies $z(\z_{\eps})=o(1)$, $z'(\z_{\eps})=o(1)$ and $Lz=f$ with $f=(L_0-L)v_{l0}=o(1)$. Using the very structure of the equation one shows that $z$ stays small in $\mathcal{C}^1$ norm for times $\z\in[\z_{\eps},\o{\z}]$, where $\o{\z}>0$ stays bounded away from zero when $(\eps,\delta,\s)\rightarrow(0,0,\s_*)$. This allows to step away from the singularity $\z=0$, and we conclude on $[\o{\z},\z_0]$ by classical regularity arguments.
\end{proof}
We have a similar result for the $\s$-derivative:
\begin{prop}
Let $z_l=\frac{\partial v_l}{\partial\s}$ and $z_{l0}=\frac{\partial v_{l0}}{\partial \s}$. For fixed $\z_0>0$ and $\s_*$ we have $z_l(\eps,\delta,\s,.)\to z_{l0}(\s_*,.)$ in $\mathcal{C}^1([\z_{\eps},\z_0])$ when $(\eps,\delta,\s)\rightarrow (0,0,\s_*)$ in the double limit \eqref{eq:regimefreqdelta}.
\label{prop:cauchystab(vg)sigma}
\end{prop}
\begin{proof}
The proof is very similar to the previous one: differentiating $Lv_l=0$ and $L_0v_{l0}=0$ with respect to $\s$ we see that $z_l$ and $z_{l0}$ solve again two very close equations, and Proposition~\ref{eq:dv'/dsigma} with $v_{l0}(0)=1$ and $v_{l0}'(0)=-(m-2)\s$ show that initial values are also close in the above limit.
\end{proof}
\par
The delicate issue for the right branches is to take into account the decay $v_r(+\infty)=0$, which turns to be a slightly more difficult problem than the previous singularity at $\z=0$ for the left branches. We first establish two decay estimates on $v_r,v_r'$ at infinity and uniformly in $\eps,\delta,\s$:
\begin{lem}
There exist $\z_0>0$ and $C>0$ locally independent of $\s$ such that, if $\eps,\delta$ are small enough in the double limit \eqref{eq:regimefreqdelta}, then $0\leq v_r(\eps,\delta,\s,\z)\leq Ce^{-\z}$ holds on $[\z_0,+\infty[$.
\label{lem:sursolutionvd}
\end{lem}
\begin{proof}
Let us recall that $v_r$ solves the elliptic equation $Lv_r=0$, with
$$
L:=-q\frac{d^2}{d\z^2}+\left(c-\frac{2q'}{m-2}\right)\frac{d}{d\z}+\left(q^{\frac{m}{m-2}}-(m-2)c\s-q''-\delta G'(\delta q)\right).
$$
The key point is that, if $\z_0$ is chosen large, the zero-th order coefficient
\begin{equation}
a_0:=q^{\frac{m}{m-2}}-\underbrace{(m-2)c\s}_{\mathcal{O}(1)}-\underbrace{q''-\delta G'(\delta q)}_{\mathcal{O}(\delta)}
\label{eq:ordre0>0}
\end{equation}
can be made larger on $[\z_0,+\infty[$ than some positive constant locally independent of $\s$: since $q(\z)\sim (m-2)c_0\z$ it is indeed enough to take $\z_0>\left(3(m-2)\s\right)^{\frac{m-2}{m}}$ and thus $a_0(\z)\geq (m-2)c_0\s\geq cst>0$ (locally in $\s$). As a consequence we may apply the classical Minimum Principle $Lv_r=0$ with $v_r(\z_0)=1$ and $v_r(+\infty)=0$ to retrieve positivity. To obtain the upper estimate, remark that at infinity the asymptotic equation associated with $Lv_r=0$ implies that $v_r$ behaves as $e^{-r\z}$ for some positive $r\underset{\eps,k,\s}{\gg} 1$: it is then easy to check that $\o{v}=Ce^{-\z}$ is a supersolution for $v_r$ if $C>0$ is chosen large enough (but independent of $\eps$ and $\delta$ and again locally of $\s$). \end{proof}
We will also need a similar estimate for the first derivative:
\begin{lem}
There exist $\z_0>0$ and $C>0$ locally independent of $\s$ such that, if $(\eps,\delta)$ are small enough in the double limit \eqref{eq:regimefreqdelta}, then $|v_r'(\z)|\leq \frac{C}{\delta^{\frac{4}{m-2}}}e^{-\z}$ holds on $[\z_0,+\infty[$.
 \label{lem:sursolutionvd'}
\end{lem}
\begin{proof}
Fix any $\z_0>0$ large enough as in Lemma~\ref{lem:sursolutionvd}: dividing \eqref{eq:Elinzetadelta} by $q>0$ we recast $Lv_r=0$ as
$$
v_r''+a v_r'=b v_r\qquad \text{and}\qquad\left\{
\begin{array}{l}
a  :=\frac{1}{q}\left(\frac{2q'}{m-2}-c\right),\\
b  := \frac{1}{q}\left(q^{\frac{m}{m-2}}-(m-2)c\s-q''-\delta G'(\delta q)\right),
\end{array}
\right.
$$
and obtain the corresponding integral formulation
$$
v_r'(\z)=-\int\limits_{\z}^{+\infty}\exp\left[A(s)-A(\z)\right]b(s)v_r(s)\mathrm{d}s
$$
with $A(s):=\int^s a$. Taking advantage of the explicit form of $a(\z)$ combined with $q(\z)\geq q(\z_0)\approx (m-2)c_0\z_0$ and $q(s)\leq q(+\infty)=\frac{1}{\delta}$ it is easy to estimate
$$
s\geq \z\geq \z_0:\quad
\left\{
\begin{array}{l}
\exp\left[A(s)-A(\z)\right]=\left(\frac{q(s)}{q(\z)}\right)^{\frac{2}{m-2}}\underbrace{\mathrm{exp}\left(-\int\limits_{\z}^s\frac{c}{q}\right)}_{\leq 1}\leq \frac{C}{\delta^{\frac{2}{m-2}}},\\
|b(s)|\leq 2 q^{\frac{m}{m-2}-1}\leq 2\left(\frac{1}{\delta}\right)^{\frac{m}{m-2}-1}
\end{array}
\right.
$$
if $\z_0>0$ is large enough. Our statement finally follows from Lemma~\ref{lem:sursolutionvd} above.
 \end{proof}
%

We have now for the right branch
\begin{prop}
For any $\s_*$ there exits $\z_0>0$ large enough such that $v_r(\eps,\delta,\s,.)\to v_{r0}(\s_*,.)$ in $\mathcal{C}^1([\z_0,+\infty[)$ when $(\eps,\delta,\s)\rightarrow (0,0,\s_*)$ in the double limit \eqref{eq:regimefreqdelta}; $\z_0$ can be chosen locally independent of $\s_*$.
\label{prop:vd->vd0}
\end{prop}
\begin{proof}
From $Lv_r=0$ and $L_0v_{r0}=0$ we see that $w:=v_r-v_{r0}$ satisfies as before $Lw=f:=(L_0-L)v_{r0}$ on $[\z_0,+\infty[$. Contrarily to the left branches the coefficients of $L$ are not close to those of $L_0$  on the whole interval $[\z_0,+\infty[$, but only on $[\z_0,\z_{\delta}]$ (this is because of our $\z$-scaling, see Proposition~\ref{prop:q->q0}). From \eqref{eq:Elinzeta} it is however easy to see that $v_{r0},v_{r0}',v_{r0}''$ decay at least exponentially at infinity, and locally uniformly in $\s_*$. According to Lemma~\ref{lem:sursolutionvd} and Lemma~\ref{lem:sursolutionvd'} we know that $v_r,v_r'$ also decay exponentially, and since $\z_{\delta}=\frac{x_{\theta}}{\delta}\underset{\eps,\delta}{\to}+\infty$ it is clearly enough to prove our statement on the subinterval $I=[\z_0,\z_{\delta}]$ .
\par
Remark that (i) the zero-th order coefficient of $L$ can be made large and positive, see \eqref{eq:ordre0>0}, (ii) $\mathcal{C}^2$ convergence of the wave profile $q_{\eps,\delta}\to q_0$ implies that $f=(L_0-L)v_{r0}$ is small in $L^{\infty}$ norm on $I$ and (iii) $w$ is exponentially small on the right boundary and vanishes on the left one, see Lemma~\ref{lem:sursolutionvd}. Using $\overline{w}:=|w|(\z_{\delta})+||f||_{\infty}=cst$ as a supersolution $L\o{w}\geq f$ we obtain $|w|\leq \overline{w}$ on $I=[\z_0,\z_{\delta}]$ and $\overline{w}=o(1)$ by construction. As a consequence $v_r$ and $v_{r0}$ are close in $L^{\infty}$ on $[\z_0,+\infty[$, and we are left to control the derivatives.
\par
According to Lemma~\ref{lem:sursolutionvd'} we have that $|v_r'|\leq \frac{C}{\delta^{\frac{4}{m-2}}}e^{-\z_{\delta}}=o(1)$ on $[\z_{\delta},+\infty[$, and $v_{r0}'$ also decays exponentially: thus $|w'|=\underset{\eps,\delta}{o}(1)$ uniformly on $[\z_{\delta},+\infty[$. In order to retrieve the same estimate on $[\z_0,\z_{\delta}]$ we integrate the equation for $w$ backward in time from $\z_{\delta}$, where $w,w'$ are small.
 \end{proof}

A similar result holds for the $\s$ derivative:
\begin{prop}
Let $z_r=\frac{\partial v_r}{\partial \s}$ and $z_{r0}=\frac{\partial v_{r0}}{\partial \s}$: for any fixed $\s_*$ there exists $\z_0>0$ such that $z_r(\eps,\delta,\s,.)\to z_{r0}(\s_*,.)$ in $\mathcal{C}^1([\z_0,+\infty[)$ when $(\eps,\delta,\s)\rightarrow (0,0,\s_*)$ in the double limit \eqref{eq:regimefreqdelta}; $\z_0$ can be chosen locally independent of $\s_*$.
\label{prop:vdsigma->vd0sigma}
\end{prop}
\begin{proof}
The proof of Proposition~\ref{prop:vd->vd0} is easily adapted, the first step being to prove two decay estimates $|z_r(\z)|\leq C e^{-\z}$ and $|z_r'(\z)|\leq \frac{C}{\delta^{\alpha}}e^{-\z}$ on $z=\frac{\partial v_r}{\partial \s}$ for some fixed exponent $\alpha=\alpha(m)>0$ (similar to Lemmas~\ref{lem:sursolutionvd} and \ref{lem:sursolutionvd'} above).
\end{proof}
%
%
\subsection{Evans function and construction of the eigenvalue}
\label{subsection:Evans}
In this section we construct the principal eigenfunction $\s=\s(\eps,\delta)$ connecting maximal decay $v_l(-\infty)=0$ and $v_r(+\infty)=0$ in the frequency regime \eqref{eq:regimefreqdelta}. Rewriting \eqref{eq:Elinzetadelta} as a first order system
$$
\frac{dY}{d\z}=A(\eps,\delta,\s;\z)Y,\qquad Y=\left(\begin{array}{c}
v\\
v' 
\end{array}
\right),
$$
we denote by $Y_l(\eps,\delta,\s,.)$ and $Y_r(\eps,\delta,\s,.)$ the previous left and right branches and define the Evans function
\begin{equation}
E(\eps,\delta,\s):=\det\big{[}Y_l(\eps,\delta,\s,\z_0),Y_r(\eps,\delta,\s,\z_0)\big{]}
\label{eq:defEvansdelta}
\end{equation}
for some $\z_0>0$ to be chosen later. Clearly the connection holds if and only if $E(\eps,\delta,\s)=0$, thus reducing the problem to finding the zeros $(\eps,\delta,\s)$ of $E$. Let us recall that the frequency open set $\omega$ is defined in \eqref{eq:regimefreqdelta}, and let us introduce
\begin{equation}
\dot{\omega}:=\omega\cup\{(0,0)\},\qquad \Omega:=\omega\times]\s_0-A,\s_0+A[,\qquad \dot{\Omega}:=\dot{\omega}\times]\s_0-A,\s_0+A[
\label{eq:defOmega}
\end{equation}
for some $A>0$. For $(\eps,\delta,\s)\in\Omega$ (hence $\eps,\delta>0$) there are no singularities in the equations: the two branches $Y_l,Y_r$ are well defined and regular in $(\eps,\delta,\s)$, and the Evans function $E(\eps,\delta,\s)$ is therefore well defined and $\mathcal{C}^1$ on $\Omega$. The set $\dot{\Omega}$ however contains critical points of the form $(0,0,\s_*)$, corresponding of course to the asymptotic problem $\eps=0,k=+\infty$.
\begin{prop}
Fix $\z_0>0$ large enough: the Evans function extends to a continuous function on $\dot{\Omega}$, again denoted by $E(\eps,\delta,\s)$. This extension is continuously differentiable with respect to $\s$ in $\dot{\Omega}$.
\label{prop:prolongementEvans}
\end{prop}
\begin{proof}
Since we consider bounded $\s\in]\s_0-A,\s_0+A[$ Propositions~\ref{prop:vg->vg0} and \ref{prop:vd->vd0} hold for some $\z_0>0$ large enough but independent of $\eps,\delta,\s$ (let us stress that all the results in Section~\ref{section:branches_regularity} were locally independent of $\s$). The branches $Y_l,Y_r$  thus continuously extend to $Y_{l0},Y_{r0}$ when $(\eps,\delta,\sigma)\overset{\Omega}{\rightarrow}(0,0,\s_*)$, and this obviously extends $E$ to $\dot{\Omega}$ setting
$$
E(0,0,\s):=\det\left(Y_{l0}(\s,\z_0),Y_{r0}(\s,\z_0)\right).
$$
By Propositions~\ref{prop:vganalytique} and \ref{prop:vdanalytique} the mapping $\s\mapsto E(0,0,\s)$ is analytical in $\s$, and Propositions~\ref{prop:cauchystab(vg)sigma} and \ref{prop:vdsigma->vd0sigma} finally imply continuity of the $\s$-derivative with respect to all three arguments $(\eps,\delta,\s)$ at any asymptotic point $(0,0,\s_*)$.
\end{proof}
According to Theorem~\ref{theo:vpprincipale} and by definition of the extension $E(0,0,.)$ we have of course
$$
E(0,0,\s_0)=0;
$$
in order to apply an Implicit Functions Theorem we need to check the usual condition:
\begin{prop}
$E$ satisfies $\frac{\partial E}{\partial \s}(0,0,\s_0)\neq 0$.
\label{prop:s0zerosimple}
\end{prop}
\begin{proof}
The argument is directly inspired from \cite{GardnerJones-topoinv}, Lemma~6.2 p.194.
\par
For $\eps=\delta=0$ the first order system corresponding to \eqref{eq:Elinzeta} is
$$
\frac{dY}{d\z}=A(\s,\z)Y,\quad Y=\left(\begin{array}{c}v\\v'\end{array}\right),\quad A(\s,\z)=\left(
\begin{array}{cc}
 0 & 1\\
\frac{b-\s}{\z} & -\frac{1}{(m-2)\z}
\end{array}
\right),
$$
and the Wronskian $W(\s,\z)=\det\left(Y_{l0}(\s,\z),Y_{r0}(\s,\z)\right)$ satisfies
\begin{equation}
\z^{\frac{1}{m-2}}W(\s,\z)=cst=\z_0^{\frac{1}{m-2}}W(\s,\z_0)=\z_0^{\frac{1}{m-2}}E(0,0,\s).
\label{eq:zetaW=zeta0E}
\end{equation}
In the light of Section~\ref{section:cadrefonctionnel} we consider $L=-\z\frac{d^2}{d\z^2}-\frac{1}{m-2}\frac{d}{d\z}+b(\z)$ as an unbounded operator $L:D(L)\subset E\rightarrow E$, of which $\s_0$ is a (principal) eigenvalue. For $\s\neq\s_0$ and $f\in E$ computing the resolvent $R(\s;L)f$ amounts to solving $-\z v''-\frac{v'}{m-2}+(b-\s)v=f$ with boundary conditions $v\in D(L)$ corresponding to $\mathcal{C}^1$ regularity at $\z=0$ and decay $v(+\infty)=0$. Using Variation of Constants and \eqref{eq:zetaW=zeta0E} this unique solution can be computed as
\begin{equation}
v=\frac{1}{\z^{\frac{1}{m-2}}W(\z)}(\alpha v_{l0}+\beta v_{r0}),
\label{eq:resolvent}
\end{equation}
with
$$
\left\{\begin{array}{rcl}
\alpha[f](\z) & := & \displaystyle{
\int_{\z}^{+\infty}t^{\frac{1}{m-2}}\Phi(t)v_{r0}'(t)dt
},\\
\beta[f](\z) & := & \displaystyle{
\int_{0}^{\z}t^{\frac{1}{m-2}}\Phi(t)v_{l0}'(t)dt
},
\end{array}
\right.
\quad\text{and}\quad
\Phi(\z) :=  \frac{1}{\z^{\frac{1}{m-2}}}\displaystyle{
\int_{0}^{\z}t^{\frac{1}{m-2}-1}f(t)dt
}.
$$
According to \eqref{eq:zetaW=zeta0E} and \eqref{eq:resolvent} the resolvent therefore reads
\begin{equation}
 \forall f\in E,\qquad R(\s;L)f=v=\frac{1}{E(0,0,\s)\z_0^{\frac{1}{m-2}}}M(\s)f,
\label{eq:defresolvente}
\end{equation}
where $M(\s)f=(\alpha v_{l0}+\beta v_{r0})$ and $\alpha[f],\beta[f]$ are defined above. By Propositions~\ref{prop:vganalytique} and \ref{prop:vdanalytique} the mappings $\s\mapsto v_{l0}(\s,.)$ and $\s\mapsto v_{r0}(\s,.)$ are analytical, and the family of operators $M(\s)$ is therefore analytical. We constructed $\s_0$ so that the principal eigenfunction $v_0=v_{l0}(\s_0,.)=v_{r0}(\s_0,.)\in D(L)$ (up to some proportionality factors): it is then the easy to compute $M(\s_0)v_0\neq 0$, and $M(\s_0)$ is consequently non-trivial.
\par
By Proposition~\ref{prop:s0vpisolee} the principal eigenvalue is isolated in $\mathbb{C}$, and the resolvent $R(\s;L)$ is thus classically meromorphic in some complex neighbourhood of $\s_0$. Since $E(0,0,\s)$ and $M(\s)$ are analytical and $M(\s_0)\neq 0$, \eqref{eq:defresolvente} shows that the order of the pole $\s=\s_0$ in $R(\s;L)$ equals the order of the zero $E(0,0,\s_0)$: by classical functional calculus arguments the order of the pole in the resolvent is exactly the algebraic multiplicity (see e.g. \cite{TaylorLay-functanal}), which is $m_a(\s_0)=1$ by Proposition~\ref{prop:multalg=1}.
\end{proof}
We may consequently apply the following Implicit Functions Theorem:
\begin{theo}
Let $\omega$ be as in \eqref{eq:regimefreqdelta}: if $\delta_0>0$ is chosen small enough there exist $B\in]0,A[$ and a function $\overline{\s}(\eps,\delta)$ defined on $\omega$ such that
\begin{enumerate}
\item 
$
\left\{\begin{array}{c}
(\eps,\delta,\s)\in\omega\times ]\s_0-B,\s_0+B[\\
E(\eps,\delta,\s)=0
       \end{array}
\right.
\quad
\Leftrightarrow
\quad
\left\{\begin{array}{c}
(\eps,\delta)\in\omega\\
\s=\overline{\s}(\eps,\delta)
       \end{array}
\right.
$
\item
$\overline{\s}$ is $\mathcal{C}^1$ on $\omega$ and $\overline{\s}(\eps,\delta)\rightarrow \s_0$ when $(\eps,\delta)\rightarrow (0,0)$ in $\omega$.
\end{enumerate}
\label{theo:TFI}
\end{theo}
Note that, due to the algebraic form \eqref{eq:regimefreqdelta} of $\omega$, taking $\delta<\delta_0$ small enough is just a vicinity condition $(\eps,\delta)\sim (0,0)$ in $\omega$, that is to say ``$(\eps,k)\in\mathcal{U}$ and $k>k_0$ large enough'' as stated in Theorem~\ref{theo:reldisp}.
\begin{proof}
This is a non-standard version of the Implicit Functions Theorem since $E(\eps,\delta,\s)$ is not $\mathcal{C}^1$ in some open set, which $\dot{\Omega}$ is not. The argument is based on the usual iterative fixed point and strongly relies on Proposition~\ref{prop:s0zerosimple}, but will not be detailed here. 
\end{proof}
%
%
\subsection{Proof of Theorem \ref{theo:reldisp}}
\label{subsection:preuvetheoreldisp}
We just defined the function $\overline{\s}(\eps,\delta)$ on $\omega$ so that $E(\eps,\delta,\overline{\s}(\eps,\delta))=0$, or in other words so that the connection $v(\pm\infty)=0$ holds in \eqref{eq:Elinzetadelta} if $\s=\o{\s}(\eps,\delta)$. We denote below by $v_{\eps,\delta}(\z)$ the corresponding eigenfunction and by $u_{\eps,k}(x)$ its alter-ego in $x$ coordinates, which is a solution of \eqref{eq:Elinx} for $\eps>0$ and $k<\infty$. Our previous scaling $\s=\frac{s}{(m-2)c_{\eps}k^{1-\frac{1}{m-1}}}$ and $\delta=\frac{1}{k^{1-\frac{1}{m-1}}}$ corresponds now to
$$
s(\eps,k):=(m-2)c_{\eps}\overline{\s}\left(\eps,\frac{1}{k^{1-\frac{1}{m-1}}}\right)k^{1-\frac{1}{m-1}},
$$
which is of course the principal eigenvalue in Theorem~\ref{theo:reldisp}. This immediately yields the desired asymptotic relation $s \underset{\eps,k}{\sim}\gamma_0k^{1-\frac{1}{m-1}}$ with $\gamma_0:=(m-2)c_0\s_0$.
%
\par
For $\s=\overline{\s}(\eps,\delta)\sim\s_0>0$ the asymptotic equation of \eqref{eq:Elinzetadelta} at $+\infty$ leads to an unstable characteristic exponent $r_+>0$ and the  stable subspace $v(+\infty)=0$ has therefore dimension $1$. This shows that the eigenspace for $s=s(\eps,k)$ is $1$-dimensional as stated in Theorem \ref{theo:reldisp}. We may have equally argued that the space of maximal decay solutions at $-\infty$ is also $1$-dimensional for $s=s(\eps,k)\sim \gamma_0k^{1-\frac{1}{m-1}}$, see Section~\ref{section:zonefroide} and in particular the discussion about $B_w$ spaces page \pageref{page:discussion_Bw}.
\par
We prove below the claimed positivity $u_{\eps,k}(x)>0\Leftrightarrow v_{\eps,\delta}(\z)>0$ separately on three intervals: in $\z$ coordinates on $[\z_{\eps},\z_0]$ (for some $\z_0$ to be chosen later), again in $\z$ coordinates on $[\z_0,+\infty[$, and finally in $x$ coordinates on $]-\infty,x_{\eps}]$.
\\
\par
\fbox{$\z\in[\z_{\eps},\z_0]$}\quad
Since $\overline{\s}(\eps,\delta)\sim \s_0$ Proposition~\ref{prop:vg->vg0} guarantees that $v_{\eps,\delta}(\z)$ is close to the asymptotic principal eigenfunction $v_{l0}(\s_0,\z)$ uniformly on $[\z_{\eps},\z_0]$ for fixed $\z_0>0$ when $(\eps,\delta)\to(0,0)$ in the double limit. By construction we had $v_{l0}(\s_0,\z)>0$ on $\R^{+}$, thus $v_{\eps,\delta}(\z)>0$ on $[\z_{\eps},\z_0]$.
\par
\fbox{$\z\in[\z_0,+\infty]$}\quad
For $\s=\overline{\s}(\eps,\delta)$ equation \eqref{eq:Elinzetadelta} is of the form $L\left[v_{\eps,\delta}\right]=0$, where $L$ is uniformly elliptic. As already discussed the zero-th order coefficient $a_0(\z)$ given by \eqref{eq:ordre0>0} can moreover be made positive on $[\z_0,+\infty[$ if $\z_0>0$ is large enough: we conclude applying the classical Minimum Principle with $v_{\eps,\delta}(\z_0)>0$ and $v_{\eps,\delta}(+\infty)=0$.
\par
\fbox{$x\in]-\infty,x_{\eps}]$}\quad
Here the classical Minimum Principle is unfortunately not available at once since the zero-th order coefficient in $Lv=0$ is not positive, see again \eqref{eq:ordre0>0} with $q(-\infty)=\eps/\delta\ll 1$. In order to keep our notations light we omit below the subscripts and one should understand in the following $v(\z)=v_{\eps,\delta}(\z)$, $u(x)=u_{\eps,k}(x)$, $p=p_{\eps}(x)$, $c=c_{\eps}$ and $s=s(\eps,k)$. The asymptotic expansion in the cold zone $v=v_0+\eps v_1+\eps v_2$ reads in $x$ coordinates on $]-\infty,x_{\eps}]$
$$
u=u_0+\eps u_1+\eps u_2,\qquad u_0=p'/p'(x_{\eps})>0,
$$
and the leading order $u_0$ is positive. According to \eqref{eq:estimationepsv1epsv2} the next orders $\eps u_1,\eps u_2$ are small on $]-\infty,x_{\eps}]$,
\begin{equation}
s\eps^{1-a}=o(1),\qquad \left\{
\begin{array}{ccl}
||\eps u_1||_{L^{\infty}} \leq C ||\eps u_1||_{B_{w,0}^2} & \leq & C_1s\eps^{1-a}\\
||\eps u_2||_{L^{\infty}} \leq C ||\eps u_2||_{B_{w,0}^2} & \leq & C_2\left(s\eps^{1-a}\right)^2
\end{array}
\right. ,
\label{eq:u>0estimationvu1u2}
\end{equation}
and $u_0>0$ should suffice to ensure positivity of the whole expansion at least for times $x\lesssim x_{\eps}$. The asymptotic study of \eqref{eq:Elinx} at $-\infty$ however shows that
$$
u\underset{-\infty}{\sim} e^{r^+x},\qquad 0<r^+=\frac{c_{\eps}+\sqrt{c_{\eps}^2+4\eps\left(k^2\eps^{\frac{m}{m-2}}-s\right)}}{2\eps}<\frac{c_{\eps}}{\eps},
$$
whereas the planar wave decays as $p'(x)\underset{-\infty}{\propto} e^{\frac{c_{\eps}}{\eps}x}\ll e^{r^+x}$, and the lower order terms $\eps u_1+\eps u_2$ unfortunately dominate $u_0$ when $x\to-\infty$. We prove below that $\eps u_1+\eps u_2$ stays negligible on some interval $[x_1,x_{\eps}]$, and use then a suitable comparison principle on $]-\infty,x_1]$.
\begin{enumerate}
 \item 
Fix $A>C_1$ with $C_1>0$ as in \eqref{eq:u>0estimationvu1u2}: monotonicity $u_0'=p''/p'(x_{\eps})>0$ and $s\eps^{1-a}\ll1$ define a unique time $x_1\in]-\infty,x_{\eps}[$ such that
\begin{equation}
u_0(x_1)=As\eps^{1-a},
\label{eq:u>0defx0}
\end{equation}
and estimate \eqref{eq:u>0estimationvu1u2} immediately implies
$$
x\in[x_1,x_{\eps}]:\qquad u\geq\underbrace{u_0}_{\geq As\eps^{1-a}}-\underbrace{||\eps u_1||_{\infty}}_{\leq C_1 s\eps^{1-a}}-\underbrace{||\eps u_2||_{\infty}}_{\leq C_2 (s\eps^{1-a})^2}>0.
$$
\item
Let $r_0:=\frac{c_0}{2\eps}>0$ and $\Phi(x):=e^{r_0x}$: using definition \eqref{eq:u>0defx0} of $x_1$ and asymptotes \eqref{eq:p'}-\eqref{eq:p''} we see that $p\sim\eps$, $p'=o(1)$ and $p''\leq C s\eps^{-a}$ hold on $]-\infty,x_1]$ and hence
\begin{eqnarray*}
L[\Phi] & = & \left[-pr_0^2+\left(c-\frac{2p'}{m-2}\right)r_0+\left(k^2p^{\frac{m}{m-2}}-s-p''\right)\right]e^{r_0x}\\
   & \geq & C\left[-\eps r_0^2+c_0 r_0-s\eps^{-a}\right]e^{r_0x}\\
 & \geq & C\left[\frac{c_0^2}{4\eps}-s\eps^{-a}\right]e^{r_0x}>0 \hspace{1cm}(s\eps^{1-a}\ll 1).
\end{eqnarray*}
The function $w:=u/\Phi$ satisfies an elliptic equation $\tilde{L}w=:Lu=0$, where $\tilde{L}$ has thus positive zero-th order coefficient $\tilde{a}_0$=$L[\Phi]>0$. Since $u\underset{-\infty}{\sim}e^{r^+x}$ decays faster that $\Phi=e^{r_0x}$ ($r^+\sim c_0/\eps>r_0$) we have $w(-\infty)=0$, and also $w(x_1)=u(x_1)/\Phi(x_1)>0$: we conclude applying the classical Minimum Principle on $]-\infty,x_1]$.
\end{enumerate}
\par
In order to fully establish Theorem~\ref{theo:reldisp} we only have left to prove that
\begin{enumerate}[(i)]
\item $s(\eps,k)$ is the smallest eigenvalue (not necessarily principal),
\item there exists no other principal eigenvalue.
\end{enumerate}

The proof of (i) is very classical and relies on the positivity of the principal eigenfunction $u$ and the fact that, the smaller $s$, the faster a potential eigenfunction must decay on both sides (this can be seen looking at the characteristic equations at $x=\pm\infty$). Assume by contradiction that there exists a non-trivial eigenfunction $\underline{u}$  associated with some $\underline{s}<s(\eps,k)$: then $w:=\underline{u}/u$ is well-defined, satisfies an elliptic equation $\tilde{L}w=0$ with zero-th order coefficient $(s-\underline{s})u>0$, and decays on both sides $w(\pm\infty)=0$. The classical Maximum/Minimum Principles finally imply $w\equiv 0$.
\par
The proof of (ii) relies on (i): if there existed an other principal eigenvalue $\o{s}>s(\eps,k)$ associated with a positive eigenfunction $\o{u}$ we could apply (i) with $\o{s}$ instead of $s$ and conclude that $\o{s}$ is the smallest eigenvalue.

%
%
%
 \bibliographystyle{siam}
 \bibliography{./biblio}
\end{document}